\documentclass[oneside,reqno]{amsart}
\usepackage{amssymb,amsmath}
\usepackage[english]{babel}
\usepackage{amsfonts}
\usepackage{color}
\usepackage{amsthm}
\usepackage[colorlinks=true,linkcolor=NavyBlue, citecolor=NavyBlue,urlcolor=NavyBlue]{hyperref}
\usepackage{graphicx}
\usepackage{mathtools}
\usepackage[usenames,dvipsnames]{pstricks}
\usepackage{bm}
\usepackage{amsmath}
\usepackage{listings}
\usepackage{hyperref}
\usepackage[shortlabels]{enumitem}

\usepackage{xy}
\usepackage[draft]{fixme}
\newcommand{\Size}[1]{\left\lvert #1 \right\rvert}
\newcommand{\Span}[1]{\left\langle#1\right\rangle}
\newcommand{\Set}[1]{\left\lbrace #1 \right\rbrace}

\newcommand{\puncture}[2]{\mathop{\vee}\!\left[#1;#2 \right]}
\let\phi\varphi
\newcommand{\bij}{f}
\DeclareMathOperator{\rk}{rk}
\theoremstyle{plain}
\newtheorem{theorem}{Theorem}[section]

\newtheorem{lemma}[theorem]{Lemma}
\newtheorem{proposition}[theorem]{Proposition}

\theoremstyle{remark}
\newtheorem{remark}{Remark}

\theoremstyle{definition}
\newtheorem{definition}[theorem]{Definition}
\newtheorem{example}[theorem]{Example}
\newtheorem*{notation*}{Notation}

\usepackage{listings}
\lstdefinelanguage{GAP}{%
 morekeywords={%
 Assert,Info,IsBound,QUIT,%
 TryNextMethod,Unbind,and,break,%
 continue,do,elif,%
 else,end,false,fi,for,%
 function,if,in,local,%
 mod,not,od,or,%
 quit,rec,repeat,return,%
 then,true,until,while%
 },%
 sensitive,%
 morecomment=[l]\#,%
 morestring=[b]",%
 morestring=[b]',%
}[keywords,comments,strings]

\usepackage[T1]{fontenc}
\usepackage[variablett]{lmodern}
\usepackage{xcolor}
\DeclareMathOperator{\N}{\mathbb{N}}
\DeclareMathOperator{\Z}{\mathbb{Z}}

\DeclareMathOperator{\Mon}{Mon}

\DeclareMathOperator{\DP}{\mathcal{O}}
\DeclareMathOperator{\Inf}{Inf}

\DeclareMathOperator{\Part}{\mathcal{P}}
\DeclareMathOperator{\Lie}{\mathfrak{L}}
\DeclareMathOperator{\gie}{\mathfrak{g}}
\DeclareMathOperator{\Sym}{Sym}

\DeclareMathOperator{\wt}{wt}

\DeclareMathOperator{\fil}{fil}
\DeclareMathOperator{\supp}{supp}

\newenvironment{nouppercase}{%
  \renewcommand{\uppercasenonmath}[1]{}}{}
  
  \usepackage{collectbox}
\makeatletter
\newcommand{\mybox}{%
    \collectbox{%
        \setlength{\fboxsep}{2pt}%
        \fbox{\BOXCONTENT}%
    }%
}

\begin{document}
\title[]{\Large{A modular idealizer chain and unrefinability of partitions with repeated parts}} 
 \author[R.~Aragona]{Riccardo Aragona}
\author[R.~Civino]{Roberto Civino}
\author[N.~Gavioli]{Norberto Gavioli}

\address{DISIM \\
 Universit\`a degli Studi dell'Aquila\\
 via Vetoio\\
 I-67100 Coppito (AQ)\\
 Italy}       

\email[R.~Aragona]{riccardo.aragona@univaq.it}
\email[R.~Civino]{roberto.civino@univaq.it} 
\email[N.~Gavioli]{norberto.gavioli@univaq.it}

\date{} \thanks{All the authors are members of INdAM-GNSAGA
 (Italy). R. Civino is funded by the Centre of excellence
 ExEMERGE at University of L'Aquila, with which also the other authors collaborate.}


\begin{abstract}
Recently Aragona et al.\ have introduced a chain of normalizers in a Sylow $2$-subgroup of $\Sym(2^n)$, starting from an elementary abelian regular subgroup. They have shown that the indices of consecutive groups in the chain depend on the number of partitions into distinct parts and have given a description, by means of rigid commutators,
of the first $n-2$ terms in the chain. Moreover, they proved that the $(n-1)$-th term of the chain is described by means of rigid commutators corresponding to unrefinable partitions into distinct parts.
Although the mentioned chain can be defined in a Sylow $p$-subgroup of $\Sym(p^n)$, for $p > 2$ computing the chain of normalizers
becomes a challenging task, in the absence of a suitable notion of rigid commutators. 
This problem is addressed here from an alternative point of view. We propose a more general framework for the normalizer chain, defining a chain of idealizers in a Lie ring over $\mathbb Z_m$ whose elements are represented by integer partitions.  
We show how the corresponding idealizers are generated by subsets of partitions into at most $m-1$ parts and we conjecture that the idealizer chain grows as the normalizer chain in the symmetric group. As an evidence of this, we establish a correspondence between the two constructions in the case $m=2$.
\end{abstract}
\begin{nouppercase}
\maketitle
\end{nouppercase}


\section{Introduction}
Let $n\ge 3$ be an integer and $\Sigma \le \Sym(2^n)$ be a Sylow $2$-subgroup containing an elementary abelian regular subgroup $T$.
Let us define $N_0 = N_\Sigma(T)$ and recursively let $N_i$ be  the normalizer in $\Sigma$ of the previous term, i.e.\ 
\begin{equation}\label{eq:chain}
N_i= N_\Sigma(N_{i-1}).
\end{equation}
Aragona et al.~\cite{Aragona2021} have recently shown that, for $1 \leq i \leq n-2$, a transversal of $N_{i-1}$ in $N_{i}$ can be put
in one-to-one correspondence with a set of partitions into distinct parts in such a way that, denoting by $\{q_{2,i}\}_{i\geq 1}$ the partial sum  of the sequence  $\{p_{2,i}\}_{i\geq 1}$ of partitions into distinct parts, the following equality is satisfied:   
\begin{equation}\label{eq:firstclaim}
\log_{2}\Size{N_{i}  : N_{i-1}} = q_{2,i+2}.
\end{equation}
The first numbers of the mentioned sequences and the relative OEIS references are displayed in Table~\ref{tab:zero}.
\begin{table}[hb]
		\centering
	{\renewcommand{\arraystretch}{1.3}
		\begin{tabular}{c||c|c|c|c|c|c|c|c|c|c|c|c|c|c|c|c||c}
			$i$ & 1 & $2$ & $3$ & $4$ & $5$ & $6$ & $7$ & $8$ & $9$ & 10& $11$ &$12$&13&14&15&16&OEIS\\
			\hline\hline
			${p_{2,i}}$ & 0 & 0& 1& 1& 2& 3& 4& 5& 7& 9& 11& 14& 17& 21& 26& 31 & \href{https://oeis.org/A111133}{A111133}\\ 
			${q_{2,i}}$ & 0& 0& 1& 2& 4&7& 11& 16& 23& 32& 43& 57& 74& 95& 121& 152 & \href{https://oeis.org/A317910}{A317910}\\
			\hline		
			\end{tabular}
		\bigskip }
	\caption{First values of the sequences $\{p_{2,i}\}$ and $\{q_{2,i}\}$}
	\label{tab:zero}
\end{table}

In a subsequent work~\cite{Aragona2022}, the authors introduced the concept of unrefinable partitions and proved that
a transversal of $N_{n-2}$ in $N_{n-1}$ is in one-to-one correspondence with a set of unrefinable partitions whose minimal excludant satisfies an additional requirement.
The study of the \emph{chain on normalizers} $(N_i)_{i \geq 0}$  has been carried out  up to the $(n-1)$-th term by means of \emph{rigid commutators}~\cite{Aragona2021}, a set of generators of $\Sigma$, which is closed under commutation and which was intentionally designed for the purpose.
However, the technique of rigid commutators could not be easily generalized to the  \emph{odd} case of the normalizer chain, i.e.\ the one defined in a Sylow $p$-subgroup of $\Sym(p^n)$, with $p$ odd. Understanding the behavior of the chain in the odd case was indeed left as an open problem by the authors.

\subsection{Overview of the new contributions}
In an attempt to achieve results in this direction, we introduce the graded Lie ring  associated to the lower central series of \(\Sigma\), which
is  the iterated wreath product of Lie rings of rank one, and reflects the construction of the Sylow \(p\)-subgroup of \(\Sym(p^n)\) (cf.\ also Sushchansky and Netreba~\cite{MR2148825}), for any prime $p \geq 2$.

More generally, given any integer $m\geq 2$, we endow the set of partitions, where each part can be repeated no more than $m-1$ times, with the  \emph{Lie ring}  structure  mentioned above. We call it the \emph{Lie ring of partitions} (cf.\ Sec.~\ref{sec:two}). In this ring we recursively define the analog of the chain of normalizers, i.e.\ the \emph{idealizer chain}, starting from an abelian subring that plays the role of the elementary abelian regular subgroup $T$.
Notice that, when $m=2$,  no part can be repeated, i.e.\ that we have the same combinatorial setting as in Aragona et al.~\cite{Aragona2021}.
Not surprisingly, we could notice that the behavior of the first $n-2$ terms of 
the chain of idealizers  is in complete accordance with that of the chain of normalizers,
 i.e.\  Eq.~\eqref{eq:firstclaim} has an analogous version for the terms of the idealizer chain, summarized in Theorem~\ref{cor:distinct_parts}.
 Interestingly, this result  can be made even more general in the setting of the Lie ring of partitions. Indeed the mentioned theorem holds in the case when $m$ is any integer greater than two, provided that partitions with at most $m-1$ repeated parts are  considered in place of partitions into distinct parts. 
 In Theorem~\ref{cor:main} we prove that  the growth of the  idealizer chain is related to the partial sums of the sequence of the number of partitions with at most $m-1$ repeated parts. This result involves the first \(n-1\) terms of the idealizer chain, one more than the case \(m=2\). 
  We conjecture that Theorem~\ref{cor:main} is the $p$-analog of the chain of normalizers in $\Sym(p^n)$, where $m=p$ is odd.

Sec.~\ref{sec:m2} is totally devoted to the case $m=2$, where we show that the terms of the normalizer chain can be actually computed via the Lie ring structure described in this paper (see Theorem~\ref{thm:liechan}). Precisely, we define a bijection (cf.\ Definition~\ref{def:bije}) from the basis elements of the Lie ring of partitions to the set of rigid commutators which preserves commutators.

In Sec.~\ref{sec:unref} we address the problem of first idealizer not following the rules of Theorems~\ref{cor:distinct_parts} and \ref{cor:main}, i.e.\ the \((n-\delta_{m,2})\)-th. If \(m=2\), it has been proved by Aragona et al.~\cite{Aragona2022} that 
$\log_{2}\Size{N_{n-1}  : N_{n-2}}$ depends on the number of a suitable subset of unrefinable partitions satisfying some additional constraints. 
We  introduce here a natural generalization of the concept of
unrefinability for  partitions with at most \(m-1\) repeated parts. We prove, in the Lie ring context, that the \((n-\delta_{m,2})\)-th idealizer  is determined by unrefinable partitions with at most \(m-1\) repeated parts satisfying the same additional constraints as in Aragona et al.~\cite{Aragona2022}  (see Theorem~\ref{prop:unrefinable}). 
 We conclude the section by giving a characterization of $n$-th idealizer (cf.\ Theorem~\ref{thm:nth}), which, by virtue of Theorem~\ref{thm:liechan}, also allows to give a precise characterization of the $n$-th normalizer \(N_n\), improving already known results~\cite{Aragona2021, Aragona2022}.

Sec.~\ref{sec:conclusion} concludes the paper with some comments on open problems.
\subsection{Related works in the combinatorics on integer partitions}
The original notion of unrefinability for partitions into distinct parts  is at least as old as the OEIS entry \href{https://oeis.org/A179009}{A179009}~\cite{OEIS} (due to David S.\ Newman in 2011) and has been formally introduced by Aragona et al.~\cite{Aragona2022}.  In that paper, unrefinable partitions satisfying a special condition on the minimum excludant appear in a natural way in connection to the chain of normalizers~\cite{Aragona2021}. The notion of minimum excludant has  been studied in the context of integer partitions by other authors~\cite{andrews2019,Ballantine2020,Hopkins2022,du2023conjecture}, although  it also appears in combinatorial game theory~\cite{gurvich2012further,fraenkel2015harnessing}.
Partial combinatorial equalities regarding unrefinable partitions have been recently shown {in~\cite{aragona2021maximal,aragona2022number},}
and the study of the algorithmic complexity of generating all the unrefinable partitions of a given integer has been addressed~\cite{aragona2021verification}.

\section{A polynomial representation of partitions of integers}\label{sec:two}
Let 
\(\Lambda = \{\lambda_i\}_{i=1}^\infty\) be a sequence of non-negative
integers  with finite support, i.e.\              such             that
\[\wt(\Lambda)=  \sum_{i=1}^\infty i\lambda_i < \infty.\] The sequence 
 \(\Lambda\) defines a  \emph{partition}  of \(N = \wt(\Lambda)\). Each non-zero $i$ is a \emph{part} of the partition, the integer \(\lambda_i\) is  the
multiplicity  of  the part  \(i\)  in  \(\Lambda\) and the support of $\Lambda$ is denoted by ${\supp(\Lambda)=\{i \mid \lambda_i \ne 0\}}$. The  maximal  part
of  \(\Lambda\)  is  the  maximum  \(i\)  such  that
\(\lambda_i\ne 0\), i.e.\ $\max \supp(\Lambda)$. The set of the partitions whose maximal part is at most
\(j\)    is     denoted    by     \(\Part(j)\)    and     we    define      for each  \(m>0\)
\[
\Part_{m}(j)= \Set{\Lambda\in\Part(j) \mid  \lambda_i\leq m-1 \text{
    for  all }  i}
    \]
  as  the set of partitions with maximal part at most $j$ and where each part has multiplicity at most $m-1$.
  We set also \[\Part_m= \bigcup_{j\ge 1}\Part_{m}(\,j).\]
	
\subsection{Power monomials}
In  the  polynomial  ring  \(\Z[x_k]_{k=1}^\infty\)  we  consider  the
monomials \(x_k^{i} \) where \(i\) is a non-negative integer.	
The  \emph{power  monomial}  \(x^\Lambda\),  where  \(\Lambda\)  is  a
partition, is defined as
\[
  x^\Lambda= \prod_{i} x_i^{\lambda_i} .
\]
These monomials clearly form a basis for \(\Z[x_k]_{k=1}^\infty\) as a
free  \(\Z\)-module.  The  set of  power monomials in at most $n$ variables  is denoted by
\[\Mon_n =   \Set{ \smash[b]{x^\Lambda \mid \Lambda  \in \Part(n)}}.\]
The  degree  of  the  power  monomials  \(x^\Lambda\)  is  defined  as
\(\deg(x^\Lambda)=\sum_{i\ge 1} \lambda_i\).
	 
Note that
\[x^\Lambda              x^\Theta             =              \prod_{i}
  x_i^{\lambda_i+\theta_i}=x^{\Lambda+\Theta}.\]  In   particular  the
\(\Z\)-module  \(\Z[x_1,\dots,x_n]\), with  basis
\(\Mon_{n}\), has  a natural  structure of  \(\Z\)-algebra and  is the
ring of polynomials in \(n\) variables with coefficients in \(\Z\).
	
The \(k\)-partial derivative is defined by
\[\partial_k(x^\Lambda) =  \begin{cases}
    0 & \text{if \(\lambda_k=0,\)}\\
    \lambda_k x^{D_k(\Lambda)} & \text{otherwise.}
  \end{cases}\]                                                  where
\(D_k(\Lambda)=\{\lambda_i-\delta_{ik}\}_{i=1}^\infty\).   In particular
\(\partial_k\)  can be  extended  by linearity  to  a derivation  over
\(\Z[x_1,\dots,x_n]\).

	Let \(m\) be a positive integer and  consider the ideal \(I=(x_1^m,\dots,x_n^m)\) of \(\Z[x_1,\dots,x_n]\). Clearly  \(\partial_{k}(I) \subseteq m\Z[x_1,\dots,x_n]\) and so  the \(k\)-th partial derivative can be seen also as a derivation defined on the \emph{ring  of power
		monomials     modulo     \(m\)     in    \(n\)     variables} (see also Strade~\cite{Strade2017})
\[
\DP_{m}(n)= \Z_{m}[x_1,\dots,x_n]/(x_1^m,\dots,x_n^m).
\]

Starting from a  modular Lie ring \(\gie\) over \(\Z\),  let us define
\(\gie^{\uparrow}=  \DP_{m}(1)\otimes_{\Z} \gie \).
We also define the \emph{inflated Lie algebra}  as
\[\Inf(\gie)= \left<\partial \otimes 1 \right>\ltimes
  \gie^{\uparrow},\]
   where \(\partial\)
  is the standard derivative.

\subsection{Lie   rings of partitions}\label{sec:part_lie_ring}

The \emph{Lie ring  \(\Lie(n)\) over \(\Z_m\) of partitions with maximal
  part at most \(n-1\)} is obtained starting from the trivial Lie ring
\(\Lie(1)= \Z_m\)          and         defining          iteratively
\(\Lie(i)= \Inf(\Lie(i-1))\).  For the sake of shortness, we shall
write \(\Lie\) in place of \(\Lie(n)\).
	
In  order   to  have  a   description  which  is  more   suitable  for
computations, \(\Lie\)  can be  seen as the  free \(\Z_m\)-module
with basis \(\mathcal{B}= \bigcup_{i=1}^n \mathcal{B}_i \), where
\[\mathcal{B}_i= \Set{ x^\Lambda \partial_i \mid
x^\Lambda\in \mathcal{O}_m(n)    \text{ with } \Lambda\in \Part_m(i-1) }. \]
	
The Lie bracket is defined on this basis by

\begin{eqnarray}\label{eq:commutator_def}
  \left[x^\Lambda \partial_k , x^\Theta \partial_j\right] =&
  \partial_{j}(x^\Lambda)  x^\Theta \partial_k - x^\Lambda \partial_{k}(x^\Theta) \partial_j= \nonumber \\
  =&
  \begin{cases}
    \partial_{j}(x^\Lambda)  x^\Theta \partial_k & \text{if \(j<k\)}, \\
    - x^\Lambda \partial_{k}(x^\Theta) \partial_j & \text{if \(j>k\)},\\
    0 & \text{otherwise,}
  \end{cases}
\end{eqnarray}
and is extended  to \(\Lie\) by bilinearity. If  \(\Lie_{i}\) is the
\(\Z_m\)-linear span  of \(\mathcal{B}_i\)  then \(\Lie_{i}\)  is an
abelian           subring           of          \(\Lie\)           and
\([\Lie_{i},\Lie_{j}]\subseteq  \Lie_{\max(i,j)}\)   and,  as  a
\(\Z_m\)-module,
\(\Lie(n)=\bigoplus_{i=1}^n\Lie_{i}=\Lie(n-1)\oplus    \Lie_{n}\). Moreover \(\Lie_{n}\) is an ideal and so \(\Lie(n)= \Lie(n-1)\ltimes    \Lie_{n}\), as a Lie ring.
\bigskip
	
For a subset \(\mathcal{H} \) of \(\Lie\) we set \(\Z_m\mathcal{H}=\Set{ax^\lambda\partial_k \mid a\in \Z_m\text{ and }x^\lambda\partial_k \in \mathcal{H}}\).
Let
\(\phi_{\Theta,j}\colon  \mathcal{B}  \to  \Z_m \mathcal{B} \)  be  the
right adjoint map defined by
\[
  \phi_{\Theta,j}\left(x^\Lambda \partial_k\right)= \left[     x^\Lambda
    \partial_k , x^\Theta\partial_j\right].
\]
	
\begin{lemma}\label{lem:injective_map}
  Let    \(    x^\Theta \partial_j     \in    \mathcal{B}\)    and
  \(\mathcal{E}=\Set{x^\Lambda   \partial_k   \in   \mathcal{B}   \mid
    \phi_{\Theta,j}\left(x^\Lambda \partial_k\right)  \ne  0   }  \).   Then the restriction
   of  \(\phi_{\Theta,j}\) to $\mathcal E$  is
  injective.
\end{lemma}
\begin{proof}
  Assume that  \(x^\Lambda \partial_k, x^\Xi  \partial_l \in \mathcal{E}\) are
 such that
  \begin{equation}\label{lem_inj_eq1}
  \phi_{\Theta,j}\left(x^\Lambda \partial_k\right)=\phi_{\Theta,j}\left(x^\Xi  \partial_l\right).
  \end{equation}  
 Since   both $\phi_{\Theta,j}\left(x^\Lambda \partial_k\right)$ and $\phi_{\Theta,j}\left(x^\Xi  \partial_l\right)$ are non-trivial,
 we either have $j > \max(k,l)$ or $j<\min(k,l)$. 
 In the first case, assuming without loss of generality that $k\leq l < j$, from Eq.~\eqref{lem_inj_eq1} we obtain 
 $
 \partial_k(x^\Theta)x^\Lambda\partial_j =  \partial_l(x^\Theta)x^\Xi\partial_j, 
 $
 i.e.\ 
 \begin{equation}\label{lem_inj_eq2}
 x^\Lambda \partial_k(x^\Theta)= x^\Xi \partial_l(x^\Theta). 
 \end{equation}
 If we assume by contradiction that $k < l$, since $\lambda_l=\xi_l=0$, we have  the exponent 
 of $x_l$ is left unchanged by the derivative $\partial_k$ in the left term of Eq.~\eqref{lem_inj_eq2}
 while it is decreased by one in the right term of Eq.~\eqref{lem_inj_eq2}, a contradiction. Hence we have $k=l$,
 from which we obtain $ x^\Lambda \partial_k(x^\Theta)= x^\Xi \partial_k(x^\Theta)$, and therefore $x^\Lambda=x^\Xi$,
 the claim.
 
In the second case, $j < \min( k,l)$ means 
\begin{equation}\label{lem_inj_eq3}
\partial_j(x^\Lambda)x^\Theta\partial_k= \partial_j(x^\Xi)x^\Theta\partial_l, 
\end{equation} 
from which immediately $k=l$. 
Then Eq.~\eqref{lem_inj_eq3} implies $\partial_j(x^\Lambda) = \partial_j(x^\Xi)$, therefore 
$x^\Lambda = x^\Xi$.
\end{proof}

\begin{definition}
A Lie subring  \(\mathfrak{H}\) of \(\Lie\) is said  to be \emph{homogeneous}
if  it  is the  \(\Z_m\)-linear  span  of  a subset  $\mathcal{H}$  of
\(\mathcal{B}\). 
\end{definition}
\begin{example}\label{defT}
The   \(\Z_m\)-submodule   \(\mathfrak{T}\)  of   \(\Lie\)          spanned       by
\(\mathcal{T}= \Set{\partial_1,\dots, \partial_{n}}\) is a homogeneous (abelian) Lie subring.  Notice that \(\partial_{i}\) is the generator of the center of \(\Lie(i)\). When \(m\) is prime, this shows that  \(\mathfrak{T}\) is the natural counterpart for the elementary abelian regular subgroup of the Sylow \(p\)-subgroup of \(\Sym(p^n)\).  
\end{example}
\begin{definition}
If \(\mathcal{H}\) is a subset  of \(\mathcal{B}\), then its \emph{idealizer}
is defined as
\[  N_\mathcal{B}(\mathcal{H})=  \Set{b\in  \mathcal{B} \mid  [b,h] \in
    \Z_m\mathcal{H} \text{ for all } h\in \mathcal{H}  }.
\]
\end{definition}

The following theorem shows that the idealizers of homogeneous subrings \(\mathfrak{H}\) can be efficiently computed directly from the intersection \(\mathfrak{H}\cap \mathcal{B}\). 
\begin{theorem}\label{thm_homogeneous:normalizers}
Let
  \(\mathfrak{H}\) be  a homogeneous subring of  \(\Lie\) having basis
  \(\mathcal{H} \subseteq   \mathcal{B}\).  The
  idealizer of \(\mathfrak{H}\) in  \(\Lie\) is
  the     homogeneous    subring     of     \(\Lie\)    spanned by
  \(N_\mathcal{B}(\mathcal{H})\) as a free \(Z_m\)-module.
\end{theorem}
\begin{proof}
Let \(\mathfrak{N}=N_{\Lie}(\mathfrak{H})\) be the
  idealizer of \(\mathfrak{H}\) in  \(\Lie\) and 
  let
  \[z=\sum_{x^\Lambda\partial_k          \in         \mathcal{B}}
    l_{\Lambda,k}x^\Lambda\partial_k  \in   \mathfrak{N}.\]  
    We
  need                  to                  show                  that
  \(l_{\Lambda,k}x^\Lambda\partial_k \in \mathfrak{N}\) for all
  \(\Lambda\)  and \(k\).   Since  \(\mathfrak{H}\)  is a  homogeneous
  subring then for all  \( x^\Theta\partial_j\in \mathcal{H}\) it
  suffices            to            show            that            if
  \(\left[l_{\Lambda,k} x^\Lambda \partial_k,
    x^\Theta\partial_j\right]\ne             0\)             then
  \(\left[l_{\Lambda,k} x^\Lambda \partial_k,
    x^\Theta\partial_j\right]           =           l_{\Lambda,k}
  \phi_{\Theta,j}\left(x^\Lambda \partial_k\right) \in \Z_m\mathcal{H}\).  Indeed,
  if \( x^\Theta\partial_j\in \mathcal{H}\), then
  \[\mathfrak{H} \ni [z, x^\Theta\partial_j] =
    \sum_{x^\Lambda\partial_k   \in  \mathcal{B}}   l_{\Lambda,k}
    \phi_{\Theta,j}\left(x^\Lambda \partial_k\right) .\]
		
  Since \(\phi_{\Theta,j}\left(x^\Lambda \partial_k\right)\in \Z_m\mathcal{B}\),
  since the  set \(\mathcal{B}\)  is a  basis  for \(\Lie\)  and since the  subset
  \(\mathcal{H}\subseteq\mathcal{B}\) is a basis for \(\mathfrak{H}\),
  by          Lemma~\ref{lem:injective_map}           we          have
  \(l_{\Lambda,k}\phi_{\Theta,j}\left(x^\Lambda \partial_k\right)   \in   \Z_m
  \mathcal{H} \) as required.
\end{proof}
\subsection{The idealizer chain}
Let us now define the bases for the \emph{chain of idealizers}, starting from the subring $\mathcal{T}$ defined in Example~\ref{defT}. 
\begin{definition}
For \(-1\le i \le n-1-\delta_{m,2}\), set
\begin{eqnarray}
  \mathcal{U}&=&\mathcal{T} \cup\Set{x_j \partial_k \mid 1\le j < k\le n}, \nonumber\\
  \mathcal{N}_i&=&
                   \begin{cases}
                     \mathcal{T} & \text{if \(i=-1\)}\\
                     \mathcal{U} & \text{if \(i=0\)}\\
                     \mathcal{N}_{i-1}\mathrel{\dot\cup} \mathcal{W}_i
                     & \text{otherwise}
                   \end{cases} \label{eq:ni}
\end{eqnarray}
where
\begin{equation}\label{eqwi}
\mathcal{W}_i = \Set{x^\Lambda \partial_k \in \mathcal{B} \mid n-i+1
    \le k \le n \text{ and } \wt(\Lambda)=k+i-n+1 + \delta_{m,2}}.
\end{equation}
\end{definition}
\begin{remark}
The need for the symbol $\delta_{m,2}$, as it will be clearer later, depends on the fact that the case \(m=2\) is different from the other cases since there is no partition of \(2\) into at least two distinct parts. 
\end{remark}

\begin{remark}\label{rem:N(n-2)}
  Note that from \eqref{eq:ni} it follows that
  \begin{align*}
    \mathcal{N}_{n-1-\delta_{m,2}} &= \Set{
                        x^\Lambda\partial_k \in \mathcal{B} \mid \wt(\Lambda) \le k 
                        },\\
    \mathcal{N}_{n-2-\delta_{m,2}} &= \Set{
                        x^\Lambda\partial_k \in \mathcal{B} \mid \wt(\Lambda) \le k-1 
                        },\\
    \intertext{and in general for \(3+\delta_{m,2}\le i\le n\)}
    \mathcal{N}_{n-i} &= \Set{
                        x^\Lambda\partial_k \in \mathcal{B} \mid \wt(\Lambda) \le k-i+1 +\delta_{m,2}
                        }\cup \mathcal{U}.
  \end{align*}
\end{remark}
\begin{definition}
The \emph{idealizer chain} starting from the \(\Z_m\)-submodule \(\mathfrak{T}\)  of   \(\Lie\) (cf.\ Example~\ref{defT})  is  defined as follows:
\begin{equation}\label{def_chain}
\mathfrak{N}_i = 
\begin{cases}
N_{\Lie}(\mathfrak{T}) & i=0,\\
N_{\Lie}(\mathfrak{N}_{i-1}) & i \geq 1.
\end{cases}
\end{equation}
\end{definition}
We will prove that for $0 \leq i \leq n-1$ the 
  Lie  subring  \(\mathfrak{N}_i\)  is  the  \(\Z_m\)-linear  span  of  \(\mathcal{N}_i\). To do so, we need the next results.

\begin{lemma}                                           \label{lem:N0}
  \(\mathcal{N}_0=N_{\mathcal{B}}(\mathcal{T})\).
\end{lemma}
\begin{proof}
We clearly have \(\mathcal{T}  \subseteq N_{\mathcal{B}}(\mathcal{T})\). Now, if \(x_{i}  \partial_j  \in \mathcal{U}\)
with $i < j$, then 
\[
\Z_m\mathcal{T} \ni [x_i\partial_j,\partial_k] = 
\begin{cases}
\partial_j & k=i,\\
0 & k \ne i,
\end{cases}
\] 
therefore \(\mathcal{U}  \subseteq  N_{\mathcal{B}}(\mathcal{T})\).

Conversely, let  \(x^\Lambda   \partial_j\in    N_{\mathcal{B}}(\mathcal{T})\). For $1 \leq k \leq n$
we have    \([x^\Lambda            \partial_j,\partial_k]=\partial_k(x^\Lambda)
  \partial_{j}\in  \Z_m\mathcal{T}\). This is possible when \(\Lambda=0\) or if
  \(x^\Lambda=x_k\) for some $1 \leq k \leq n$, i.e.\ $x^\Lambda \partial_j \in \mathcal{N}_0$.
\end{proof}

The following result, which will be useful later on,   is straightforward.
\begin{lemma}                \label{lem:weight_comm}                If
  \([x^\Lambda\partial_j, x^\Theta\partial_k]= c x^\Gamma\partial_u\),
  where    \(0\ne    c\in    \Z_m\),    then    \(u=\max(j,k)\)    and
  \(\wt(\Gamma)=\wt(\Lambda)+\wt(\Theta)-\min(j,k)\).
\end{lemma}

\begin{lemma}
If \(1\le i \le n-1-\delta_{m,2}\), then  \([\mathcal{U},\mathcal{W}_i]   \subseteq  \Z_m\mathcal{N}_{i-1}\).
\end{lemma}
\begin{proof}
	 
  Let        \(x^\Lambda\partial_j\in       \mathcal{W}_i\)        and
  \(x_h^{e_h}\partial_k\in \mathcal{U}\), where \(0\le e_h\le 1\) and let 
  $ c
  x^\Gamma\partial_u = [x^\Lambda\partial_j,            x_h^{e_h}\partial_k]$,    where     \(0\ne    c\in     \Z_m\).     If
  \(x^\Gamma\partial_u  \in  \mathcal{U}\)  there is  nothing  to 
  prove, so assume    \(x^\Gamma\partial_u    \notin
  \mathcal{U}\). If \(k\le j\),  then       either       \(c=0\)      or, since $\wt(\Gamma)<\wt(\Lambda)$,
  \(  x^\Gamma\partial_u \in  \mathcal{N}_{i-1}\).  Otherwise, if
  \( k>j \), then 
  $ c x^\Gamma\partial_u = \partial_j (x_h^{e_h})x^\Lambda\partial_k \ne 0$ if and only if 
  $h=j$ and $e_h=1$. Moreover, since we are assuming \(x^\Gamma\partial_u    \notin
  \mathcal{U}\), then it satisfies Eq.~\eqref{eqwi}, and we have 
  \[
 \wt(\Lambda) = j+1-(n-1)+\delta_{m,2}.
 \]
  Now, $c x^\Gamma\partial_u = x^\Lambda\partial_k$ and 
  \begin{eqnarray*}
  \wt(\Gamma) =   \wt(\Lambda)  &=&j+i-(n-1)+\delta_{m,2}\\
  &\leq& k+i-1-(n-1)+\delta_{m,2},
  \end{eqnarray*}
 therefore $x^\Gamma\partial_u \in  \mathcal{N}_{i-1}$.
\end{proof}

\begin{lemma}
  If      \(1\le       i<h      \le       n-1-\delta_{m,2}\)      then
  \([\mathcal{W}_i, \mathcal{W}_h] \subseteq \Z_{m}\mathcal{N}_{h-1}\).
\end{lemma}
\begin{proof}
  Let       \(x^\Lambda\partial_j\in        \mathcal{W}_i\) and        \(x^\Theta\partial_k\in              \mathcal{W}_h\). Let us denote             
  \([x^\Lambda\partial_j,           x^{\Theta}\partial_k]=           c
  x^\Gamma\partial_u\)        with        \(c\ne        0\)        and let us assume that
  \(x^\Gamma\partial_u\notin             \mathcal{U}\) otherwise, as before, there is nothing to prove.
  By  \(x^\Lambda\partial_j\in        \mathcal{W}_i\) we obtain $\wt(\Lambda) = j+i-(n-1)+\delta_{m,2}$
   and by \(x^\Theta\partial_k\in        \mathcal{W}_h\) we obtain  $\wt(\Theta) = k+h-(n-1)+\delta_{m,2}$.
  Now, by Lemma~\ref{lem:weight_comm}  we have 
  \begin{eqnarray*}
   \wt(\Gamma)&=&\wt(\Lambda)+\wt(\Theta)-\min(j,k)\\
   &=&j+i-(n-1)    + \delta_{m,2} + k+h-(n-1) +  \delta_{m,2} -\min(j,k)\\
   &=& \max(j,k)+i-(n-1) +  \delta_{m,2} +h-(n-1) + \delta_{m,2}\\
   &=& u+h-(n-1) + \delta_{m,2} +i-n+1 +  \delta_{m,2} \\
   &\leq& u+(h-1)-(n-1) + \delta_{m,2},
  \end{eqnarray*}
  which implies   \(x^\Gamma\partial_u\in \mathcal{N}_{h-1}\).
\end{proof}

\begin{proposition}\label{prop:Ni}
If  \(1 \leq i \leq n-1-\delta_{m,2}\), then 
  \(\mathcal{N}_i=N_{\mathcal{B}}(\mathcal{N}_{i-1})\).
\end{proposition}
\begin{proof}
 The inclusion $\mathcal{N}_i \subseteq N_{\mathcal{B}}(\mathcal{N}_{i-1})$ follows from the previous lemmas. It
 remains to prove that $N_{\mathcal{B}}(\mathcal{N}_{i-1})\subseteq \mathcal{N}_i$. Let  
 $x^\Lambda\partial_j \in N_{\mathcal{B}}(\mathcal{N}_{i-1})$. Then for each $1 \leq l \leq n-1-\delta_{m,2}$
 and for each \(x^\Theta \partial_k \in \mathcal{W}_l\) we have 
 \([x^\Lambda  \partial_j  ,   x^\Theta  \partial_k] \in \mathcal{N}_{i-1} \setminus \mathcal{N}_{0}.\)
 Let \(k<j\)  be minimum
        such      that      \(\lambda_k\ne      0\),   
        and let $x^\Theta\partial_k=x_{k-1}\partial_k$.
        Then, since $\lambda_k \ne 0$, we have $[x^\Lambda\partial_j,x_{k-1}\partial_k] \ne 0$ and, by hypothesis,
        $[x^\Lambda\partial_j,x_{k-1}\partial_k]$ is a scalar multilple of an element 
        $x^\Gamma\partial_j \in \mathcal{N}_{i-1}$ such that 
\begin{eqnarray*}
\wt(\Gamma) = \wt(\Lambda) -1 &\leq& j+i-1-(n-1)+\delta_{m,2}\\
&<& j+i-(n-1)+\delta_{m,2}.
\end{eqnarray*}
 Therefore $\wt(\Lambda) \leq j+i-(n-1)+\delta_{m,2}$, i.e.\ $x^\Lambda\partial_j \in \mathcal{N}_i$.
\end{proof}
Based on the previous Lemma for all \(i\in \N \) we may define
\begin{equation}                            \label{eq:normalizer_comm}
  \mathcal{N}_i= N_{\mathcal{B}}(\mathcal{N}_{i-1})
\end{equation}

\begin{theorem}\label{thm:homogeneous_normalizers}
  The Lie subring \(\mathfrak{N}_i\)  is homogeneous and 
   the
  Lie  subring  \(\mathfrak{N}_i\)  is  the  \(\Z_m\)-linear  span  of
  \(\mathcal{N}_i\).  
\end{theorem}
\begin{proof}
  The    statement    is     a    straightforward    consequence    of
  Theorem~\ref{thm_homogeneous:normalizers} and  of Lemma~\ref{lem:N0}
  and Proposition~\ref{prop:Ni}.
\end{proof}

\subsection{Connections with integer partitions}
Let $p_{m,i}$ be the number of partitions of $i$ into at least two parts, where each part can be repeated at most $m-1$ times, 
and let $q_{m,i}$ be the partial sum \[q_{m,i}=\sum_{j=1}^ip_{m,j}.\] The first values of the sequences are showed in Tab.~\ref{tab:one}. Notice that the last three OEIS entries of the table include the partition of $i$ with a single part that we do not consider.

\begin{table}[]
		\centering
	{\renewcommand{\arraystretch}{1.3}
		\begin{tabular}{c||c|c|c|c|c|c|c|c|c|c|c|c|c|c|c|c||c}
			$i$ & 1 & $2$ & $3$ & $4$ & $5$ & $6$ & $7$ & $8$ & $9$ & 10& $11$ &$12$&13&14&15&16&OEIS\\
			\hline\hline
			${p_{2,i}}$ & 0 & 0& 1& 1& 2& 3& 4& 5& 7& 9& 11& 14& 17& 21& 26& 31 & \href{https://oeis.org/A111133}{A111133}\\ 
			${q_{2,i}}$ & 0& 0& 1& 2& 4&7& 11& 16& 23& 32& 43& 57& 74& 95& 121& 152 & \href{https://oeis.org/A317910}{A317910}\\
			\hline
			${p_{3,i}}$ & 0& 1& 1& 3& 4& 6&8& 12& 15& 21& 26& 35& 43& 56& 69& 88 & \href{https://oeis.org/A000726}{A000726}\\  
			${q_{3,i}}$ & 0& 1& 2& 5& 9& 15& 23& 35& 50& 71& 97& 132& 175& 231& 300& 388\\ 
			\hline 
			${p_{4,i}}$ & 0& 1& 2& 3& 5& 8& 11& 15& 21& 28& 37& 49& 63& 81& 104& 131&  \href{https://oeis.org/A001935}{A001935}\\
			${q_{4,i}}$ & 0& 1& 3& 6& 11& 19& 30& 45& 66& 94& 131& 180& 243& 324& 428& 559&  \\ 
			\hline
			${p_{5,i}}$ & 0& 1& 2& 4& 5& 9& 12& 18& 24& 33& 43& 59& 75& 99& 126& 163&   \href{https://oeis.org/A035959}{A035959}\\
			${q_{5,i}}$ & 0& 1& 3& 7& 12& 21& 33& 51& 75& 108& 151& 210& 285& 384& 510& 673&   \\
			\hline		\end{tabular}
		\bigskip }
	\caption{First values of the sequences $(p_{m,i})$ and $(q_{m,i})$ for $2 \leq m \leq 5$}
	\label{tab:one}
\end{table}
\medskip

From Theorem~\ref{thm:homogeneous_normalizers} we derive the following  corollaries, 
here stated in the case $m=2$ and $m>2$ separately.

\begin{theorem}\label{cor:distinct_parts}
  Let      \(m=2\)      and      \(1\le      i\le      n-2\).      Then, for $n-i+1 \leq k \leq n$ we have
  \(\Size{\mathcal{W}_i\cap  \mathcal{B}_k}=  p_{2,k+2+i-n} \)  and therefore the
  free  \(\Z_2\)-module   \(\mathfrak{N}_{i}/\mathfrak{N}_{i-1}\)   has 
  rank \(q_{2,i+2}\).
\end{theorem}

Notice that the result of Theorem~\ref{cor:distinct_parts} is in complete
accordance with  the analogous result found  in the case of  the chain of
normalizers in the Sylow \(2\)-subgroup of \(\Sym(2^n)\) starting from
an   elementary   abelian   regular  subgroup   (\cite[Corollary
5]{Aragona2021}). This is not surprising: we will indeed prove in Sec.~\ref{sec:m2} that there exists a correspondence 
between the two constructions.
More importantly, the use of the Lie ring of partitions introduced here allows 
to easily generalize the  result to the case $m>2$.

\begin{theorem}\label{cor:main}
  Let      \(m>2\)      and      \(1\le      i\le      n-1\).      Then, for $n-i+1 \leq k \leq n$ we have
  \(\Size{\mathcal{W}_i\cap  \mathcal{B}_k}=  p_{m,{k+1+i-n}} \)  and therefore the
  free  \(\Z_m\)-module   \(\mathfrak{N}_{i}/\mathfrak{N}_{i-1}\)  has
  rank \(q_{m,i+1}\).
\end{theorem}

\section{An explicit correspondence in the case \(m=2\)}\label{sec:m2}
 In this section we will assume $m=2$ without further reference. As already anticipated above, we now prove that \emph{for any $i\geq 1$} the ranks of the quotients \(\mathfrak{N}_i/\mathfrak{N}_{i-1}\) are equal to the logarithms $\log_{2}\Size{N_{i}  : N_{i-1}}$ of the factors of the normalizer chain in the Sylow \(2\)-subgroup of \(\Sym(2^n)\) starting from an elementary abelian regular subgroup. This is constructively accomplished by showing a bijection which maps rigid commutators into basis elements of the Lie ring of partitions and which preserves commutators.

\subsection{Correspondence with Sylow \(2\)-subgroups of $\Sym(2^n)$}	\label{sub:corresp}  
We recall here some fundamental facts about rigid commutators, although we advise the reader to refer to Aragona et al.~\cite{Aragona2021} for notation and results. We use the punctured notation as in the mentioned paper. More precisely, if \(\Set{s_1, s_2, \dots,s_n}\) is the considered set of generators of the Sylow \(2\)-subgroup of \(\Sym(2^n)\) and \(X=\Set{x_1> x_2>\dots > x_\ell}\) is a subset of \(\Set{1,\dots, n}\), we denote by \([X]\) the  left normed commutator \([s_{x_1},s_{x_2},\dots,s_{x_\ell}]\). 
The \emph{rigid commutator based at $b$ and punctured at $I$} is 
\begin{equation*}
	\puncture{b}{I} = [  \Set{1,\dots,b} \setminus
	I] \in \mathcal{R}^*,
\end{equation*}
where \(1\le b \le n\) and \(I\subseteq\Set{1,\dots,b-1}\) and the symbol \(\mathcal{R}^*\) denotes  the set of non-trivial rigid commutators.  We also denote \(\mathcal{R}= \mathcal{R}^*\cup \Set{[\emptyset]}\).
We will use the commutator formula
\begin{equation}\label{eq:puncture}
	\bigl[\,
	\puncture{a}{I} ,
	\puncture{b}{J}
	\, \bigr] = 
	\begin{cases}
		\puncture{\max (a,b)}{\;  (I \cup  J)\setminus\Set{\min(a,b)}} &
		\text{if $\min (a,b)\in I \cup J$}
		\\
		1 & \text{otherwise}
	\end{cases}
\end{equation}
proved in Proposition~4 of the referenced paper. We also remind  that the elementary abelian regular group $T$ is obtained in terms of rigid commutators as \(T=\Span{t_1,\dots,t_n}\), where \(t_i=\puncture{i}{\emptyset}\) for $1 \leq i \leq n$.\\

The mentioned bijection that will be soon defined relies crucially on the representation of rigid commutators provided by the following result.
\begin{lemma}\label{lem:disjoint}
	Let \(\mathcal{S}\subseteq \mathcal{R}\) be normalized by \(\Set{t_1,\dots,t_n}\). If \(\puncture{a}{X}\) is any rigid commutator normalizing \(\mathcal{S}\) and \(\puncture{b}{Y}\in \mathcal{S}\), then there exists a rigid commutator \(\puncture{b}{Z}\in  \mathcal{S}\) such that \(Z\cap X=\emptyset\) and \[\bigl[\puncture{a}{X}, \puncture{b}{Y}\bigr]= \bigl[\puncture{a}{X}, \puncture{b}{Z}\bigr].\]
\end{lemma}
\begin{proof}
	Let \(i\in X\cap Y\). Note that \[\bigl [\puncture{a}{X},\puncture{b}{Y\setminus\Set{i}}\bigr]= \bigl [\puncture{a}{X},[\puncture{b}{Y},t_i]\bigr] = \bigl[\puncture{a}{X},\puncture{b}{Y}\bigr]. \] In this way we can remove one by one from \(Y\) all the elements in \(X\cap Y\) obtaining \(Z\) and preserving the commutator.
\end{proof}
Let us now define the bijection $f$ between basis elements of the Lie ring and the set of rigid commutators. We will show later that $f$ preserve commutators.
\begin{definition}\label{def:bije}
Let  \(\bij\colon \mathcal{B} \cup \Set{0} \to \mathcal{R} \) be defined by letting \(\bij(0)=[\emptyset]\) and
\[
f(x^\Lambda\partial_k)=\puncture{k}{ \supp(\Lambda)}.
\]
\end{definition}
	\begin{remark}\label{rem:comm}
		By Eq.~\eqref{eq:puncture} we have that if either \([x^\Lambda\partial_k,x^\Gamma\partial_h]\ne 0\) or \(\Lambda\cap \Gamma=\emptyset\), then 
		\[f\left([x^\Lambda\partial_k,x^\Gamma\partial_h]\right)= \bigl[f\left(x^\Lambda\partial_k\right),f\left(x^\Gamma\partial_h\right)\bigr]. \]
		We note indeed that if \([x^\Lambda\partial_k,x^\Gamma\partial_h]= 0\) and  \(\Lambda\cap \Gamma=\emptyset\), then \(k\notin \Gamma\) and \(h\notin \Lambda\) and hence both members of the previous equation are the identity element.
	\end{remark}
\begin{lemma}
	If \(S\subseteq \mathcal{B}\cup \{0\}\) is normalized by \(\mathcal{T}=\Set{\partial_1,\dots, \partial_{n}}\) and is closed under commutation, then \(x^\Theta\partial_u\) normalizes \(S\) if and only if \(\bij(x^\Theta\partial_u)\) normalizes \(\mathcal{S}=\bij(S)\).
\end{lemma}
\begin{proof}
	We show first that \(\mathcal{S}\) is closed under commutation. 
Notice that, since \([\mathcal{T},S]\subseteq S\), by Remark~\ref{rem:comm}  we have that \(\bij(\mathcal{T})=\Set{t_1,\dots,t_n}\) normalizes \(\mathcal{S}\).
	Let \(\bij(x^\Lambda\partial_k)\) and \(\bij(x^\Gamma\partial_h)\) be two elements in \(\mathcal{S}\). By Lemma~\ref{lem:disjoint}, and by Remark~\ref{rem:comm} we have
	 \[\left[\bij(x^\Lambda\partial_k),\bij(x^\Gamma\partial_h)\right]= 
	\left[\bij(x^\Lambda\partial_k),\bij(x^{\Gamma'}\partial_h)\right]= \bij\left([x^\Lambda\partial_k,x^{\Gamma'}\partial_h]\right)\in \mathcal{S}\] for some \(\Gamma'\) such that \(\supp(\Lambda)\cap \supp({\Gamma'})=\emptyset\). 
	
	Let \(x^\Lambda\partial_k\in S\) and \(x^\Theta\partial_u\) be a basis element in the Lie ring normalizing \(S\). The commutator \([x^\Theta\partial_u, x^\Lambda\partial_k]\in S\), hence, by Lemma~\ref{lem:disjoint}, there exists \(\Lambda'\) such that \(\supp(\Lambda')\cap \supp(\Theta)=\emptyset\) and \[[\bij(x^\Theta\partial_u),\bij(x^\Lambda\partial_k)]=[\bij(x^\Theta\partial_u),\bij(x^{\Lambda'}\partial_k)]=\bij\bigl([x^\Theta\partial_u,x^{\Lambda'}\partial_k]\bigr)\in \mathcal{S}.\] 
	Therefore \(\bij(x^\Theta\partial_u)\) normalizes \(\mathcal{S}\). Conversely, if \(\bij(x^\Theta\partial_u)\) normalizes \(\mathcal{S}\) and \(\bij(x^{\Lambda'}\partial_k)\in \mathcal{S}\), then \( [\bij(x^\Theta\partial_u),\bij(x^\Lambda\partial_k)] \in \mathcal{S}
	\). Thus either \( [x^\Theta\partial_u,x^\Lambda\partial_k]=0 \in S
	\) or \[[\bij(x^\Theta\partial_u),\bij(x^\Lambda\partial_k)] = \bij\left([x^\Theta\partial_u,x^\Lambda\partial_k]\right) \in \mathcal{S}=\bij(S).\] Hence \([x^\Theta\partial_u,x^\Lambda\partial_k]\in S\), as \(\bij\) is a bijection, and so \(x^\Theta\partial_u\) normalizes \(S\). 
\end{proof}
	We are finally ready to prove the claimed result.
	\begin{theorem}\label{thm:liechan}
		For all non-negative integers \(i\) 
		the term \(N_i\) of the normalizer chain is the saturated subgroup generated by the saturated set of rigid commutators \(\bij(\mathcal{N}_i)\).
		In particular, the following equality holds for each $i \geq 1$:
		\[\rk\left(\mathfrak{N}_i/\mathfrak{N}_{i-1}\right)=\log_{2}\Size{N_{i}  : N_{i-1}}.\]
	\end{theorem}
\begin{proof}
		This is a straightforward consequence of the previous lemma applying  Theorem~\ref{thm:homogeneous_normalizers} and Corollary~2 and Proposition~5 from Aragona et al.~\cite{Aragona2021}.
\end{proof}

\section{Unrefinable partitions with repeated parts and the $(n-1)$-th idealizer}\label{sec:unref}

The definition of unrefinability of a partition into distinct parts has been given in Aragona et al.~\cite{Aragona2022} in connection with the $(n-1)$-th term in the chain of normalizers in $\Sym(2^n)$. We introduce here a natural generalization to partitions whose parts can be repeated at most \(m-1\) times and we show the connection (cf.~Theorem~\ref{prop:unrefinable}) with the first idealizer not following the rules of Theorems~\ref{cor:distinct_parts} and \ref{cor:main}, i.e.\ the \((n-\delta_{m,2})\)-th.

\begin{definition}\label{def:unref}
  Let \(\Lambda\in \Part_m\) be  a partition where each part  has multiplicitiy at most \(m-1\) and such that there
  exist indices \(j_1<\dots<j_\ell<j\) satisfying the conditions
  \begin{itemize}
  \item \(j=\sum_{i=1}^\ell a_ij_i\), with \(a_i\le m-1-\lambda_{j_i}\),
  \item \(\lambda_j\ge 1\).
  \end{itemize}
  The partition \(\Theta\) obtained from \(\Lambda\) removing the part
  \(j\)           and           inserting          the           parts
  \({j_1},\dots  ,  {j_\ell}\),  each taken  \(a_i\) times, is   said  to  be  an
  \(a\)-refinement    of     \(\Lambda\) where \(a=\sum a_{i}\).     We     shall    write
  \(\Theta\prec   \Lambda\)    to   mean   that   \(\Theta\)    is   a
  \(2\)-refinement of \(\Lambda\). A partition admitting a refinement is said to be \emph{refinable in~\(\Part_{m}\)}, otherwise it is said to be \emph{unrefinable in~\(\Part_{m}\)}.
\end{definition}
\begin{remark}
Notice that, although the part $j$ can appear with multiplicity up to $m-1$, the operation of refinement as in Definition~\ref{def:unref} is performed on a single part.
\end{remark}

\begin{proposition}\label{prop:refinement_length}
  Every  \(a\)-refinement of  a partition  \(\Lambda\) is  obtained
  applying exactly \(a-1\) subsequent \(2\)-refinements.
\end{proposition}
\begin{proof}
  Let    \(j\)   be    the   part    of   \(\Lambda\)    replaced   by
  \(a_1\) repetitions of \(j_1\), \dots , and \(a_\ell\) repetitions of \(j_\ell\).  We  split the proof in  two cases, depending
  on \(\lambda_{j_1+j_2}\ge 1 \) or \(\lambda_{j_1+j_2}=0\), and we argue by induction,
  the    statement    being     trivial    when    \(a=2\).     Let
  \(\lambda_{j_1+j_2}\ge 1\).  First  we apply the  \(2\)-refinement that
  inserts \(j_1\) and \(j_2\) in place of \(j_1+j_2\). Subsequently we
  apply the  induction argument on  the refinement replacing  \(j\) by
  inserting    \(j_1+j_2,    j_3,\dots,   j_\ell\) via \(a-2\) subsequent \(2\)-refinements.     Suppose    now
  \(\lambda_{j_1+j_2}=0\).     We      first     apply     the
  \((a-2)\)-refinement      replacing     \(j\)      by     inserting
  \(j_1+j_2,  j_3,\dots,  j_\ell\)  and   subsequently  we  apply  the
  \(2\)-refinement  that  inserts  \(j_1\)  and \(j_2\)  in  place  of
  \(j_1+j_2\).   In both  cases by  induction a  number \(a-1\)  of
  \(2\)-refinement are applied. Since every \(2\)-refinement increases
  by one  the total number  of the parts,  \(a -1\) is  the minimum
  possible  number  of  \(2\)-refinements  that  we  can  subsequently
  perform to obtain the final \(a\)-refinement.
\end{proof}

\begin{definition} 
  Let \(\Lambda\in \Part_{m}\)  and \(t> 0\) be an  integer.  We say
  that \(\Lambda\)  is \(0\)-step refinable  if it is  unrefinable in \(\Part_m\). We
  say that  \(\Lambda\) is  \(t\)-step refinable  if \(t\)  is maximal
  such that there exists a a  sequence made of \(t\) subsequent proper
  \(2\)-refinements
  \(\Lambda_t \prec  \Lambda_{t-1} \prec\dots\prec \Lambda_0=\Lambda\) such
  that  \(\Lambda_t\) is  unrefinable.  In  other words  \(t\) is  the
  maximum number of \(2\)-refinements to be subsequently applied 
  starting  from  \(\Lambda\)  in   order  to  obtain  some  
  partition that is unrefinable in \(\Part_m\).
\end{definition}

\begin{remark}
	A straightforward consequence of Proposition~\ref{prop:refinement_length} is that a partition \(\Lambda\) in \(\Part_m\) is \(t\)-step refinable if and only if \(t\) is maximal among the \(a\) such that \(\Lambda\) admits an \(a\)-refinement. 
\end{remark}

\begin{definition}
  Let    \(\Lambda\in    \Part_{m}(n-1)\).    Consider    the    monomial
  \(f= \prod_{i=1}^{n-1}x_i^{m-1}\)                    and                   let
  \[x_{e_1}^{\mu_1}\cdots     x_{e_s}^{\mu_s}=    f/x^\Lambda,\]    where
  \(e_1 < \dots < e_s\) and \(\mu_i\ge 1\). The index  \(e_i\) is said to be the \emph{\(i\)-th
  excludant}  of \(\Lambda\) and \(\mu_i\) is its multiplicity.
  The first excludant of  \(\Lambda\) is also called its \emph{minimum excludant}.
     We  say that  \(x^\Lambda\partial_{k}\)
  satisfies the \emph{\(i\)-th  excludant condition} if \(i\)  is the minimum
  index    such    that    \(n    <   k+e_{i}\).    Moreover, we    say    that
  \(x^\Lambda\partial_{k}\)  satisfies  the  \emph{weak  \(i\)-th  excludant
  condition}   if    \(i\)   is    the   minimum   index    such   that
  \(n < k+e_1+\dots +e_{i}\).
\end{definition}
Note that if  a partition satisfies the \(i\)-th excludant condition  then it also
satisfies the weak \(j\)-th excludant condition for some \(j\le i\).

\medspace

We define the \emph{filler element} as
\begin{equation}
  \fil_{i,j}= x_ix_j\partial_{i+j} \in \mathfrak{N}_{n-1-\delta_{m,2}} \setminus \mathfrak{N}_{n-2-\delta_{m,2}}.
\end{equation}
Let  \(\Lambda\in   \Part_{m}(n-1)\)  be  a  partition   with  excludants
\(e_1<       \dots      <       e_s\)      and       suppose      that
\(x^\Lambda\partial_{k} \in \mathfrak{N}_j\) for some \(j\ge n-\delta_{m,2}\). If
\(k+e_i \le n\), then the commutator operation
\begin{equation*}
  [x^\Lambda\partial_{k}, 	\fil_{e_i,k}] = x_{e_i}x^\Lambda\partial_{k+e_i} \in \mathfrak{N}_{j-1} 
\end{equation*}
has   the  effect   of  \emph{filling}   the  \(i\)-th   excludant  of
\(\Lambda\).
\medskip

We now deal  with the main result of the section. The condition for a  partition \(\Lambda\in \Part_m(k-1)\) to be
refinable  is equivalent  to the  fact that  there exists  a partition
\(\Theta\in \Part_m(h-1)\)       with       \(h=\wt(\Theta)<k\),       such       that
\( [x^\Lambda\partial_k, x^\Theta \partial_h]\ne 0\). 

\begin{theorem}\label{prop:unrefinable}
  The       elements      of      the      set
  \(\mathcal{N}_{n-\delta_{m,2}}\setminus   \mathcal{N}_{n-1-\delta_{m,2}}\)  are  of the form  \(x^\Lambda\partial_k \in \mathcal{B}\),  where \(\Lambda\in \Part_{m}(n-1)\) is an
  unrefinable partition of \(k+1\)  satisfying the first excludant condition. 
\end{theorem}

\begin{proof}
	We prove the claim  assuming \(m>2\). The proof of the case \(m=2\) is nearly identical, and also unnecessary, by virtue of the correspondence shown in Sec.~\ref{sec:m2}.
	
  Let
  \(x^\Lambda\partial_k\in         \mathcal{N}_{n}         \setminus
  \mathcal{N}_{n-1}\) and let \(e\) be the minimal excludant of \(\Lambda\).    By    Remark~\ref{rem:N(n-2)}, since  \(x^\Lambda\partial_k\not\in \mathcal{N}_{n-1}\),   we    have
  \(\wt(\Lambda)\ge  k+1\). Let  \(h=\min\Set{j\mid \lambda_j\ne  0}\)
  and let \[ \mathcal{N}_0\ni x^\Gamma\partial_h= \begin{cases}
      \partial_1 & \text{if \(h=1\)} \\
      x_{h-1}\partial_h & \text{if \(h>1\)}
    \end{cases}
  \]
  Since
  \(  \mathcal{N}_{n-1}\ni  [x^\Gamma\partial_h,  x^\Lambda\partial_k]=
  x^\Theta\partial_k\ne        0\)        it       follows        that
  \(\wt(\Theta)=\wt(\Lambda)-1\le k\). Hence \(\wt(\Lambda)= k+1\).
	
  Let \(\Xi\) be any partition of weight \(k+1\),  by   Lemma~\ref{lem:weight_comm} 
  \([x^\Xi\partial_k,\mathcal{N}_{n-2}]\subseteq
  \mathcal{N}_{n-1}\), again by   Lemma~\ref{lem:weight_comm}  and
  Remark~\ref{rem:N(n-2)}       it       follows       that       
  \(x^\Xi\partial_k\in         \mathcal{N}_{n}         \setminus
  \mathcal{N}_{n-1}\)          if         and          only         if
  \([x^\Xi\partial_k,\mathcal{W}_{n-1}]=0\).        Let       then
  \(x^\Sigma\partial_h\in      \mathcal{W}_{n-1}\),       so      that
  \(\wt(\Sigma)=h\).                    The                  condition
  \([x^\Sigma\partial_h,     x^\Lambda\partial_k]=0\)      for     all
  \(x^\Sigma\partial_h\in  \mathcal{W}_{n-1}\setminus \mathcal{U}\)  with   \(h\le  k\)  is
  equivalent  to  \(x^\Lambda\partial_k\)  being unrefinable.   So  we
  assume                          \(h>k\)                          and
  \(x^\Theta\partial_h= [x^\Sigma\partial_h,   x^\Lambda\partial_k]\ne
  0\). In particular \(\wt(\Sigma) \ge e+k\), since \(\sigma_k\ge 1\) and
 since \(\Sigma\)  can  have  non-zero   components  \(\sigma_i\)  only  if
  \(i\ne    k\)   is    an    excludant    of   \(\Lambda\).     Hence
  \(n\ge  h=  1+\wt(\Sigma)  \ge e+k\)  yielding  \(k\le  n-e\).
  Conversely        if        \(0<       k\le        n-e\)        then
  \(\fil_{e,k}\in                \mathcal{N}_{n-1}\)               and
  \([x^\Lambda\partial_k,\fil_{e,k}]=x_ex^\Lambda\partial_{e+k}\ne
  0\).       Hence       if      \(\wt(\Lambda)=       k+1\)      then
  \([x^\Lambda\partial_k,\mathcal{W}_{n-1}\setminus \mathcal{U}]=0\)   if   and   only   if
  \(n-e < k \le n\) and \(x^\Lambda\partial_k\) is unrefinable in \(\Part_{m}\).
\end{proof}

\subsection{One more idealizer}

In  this  last section we set again $m=2$ and we aim at the characterization of the \(n\)-th term  of the idealizer chain defined in Eq.~\eqref{def_chain}. By virtue of the results of Sec.~\ref{sub:corresp}, the characterization automatically extends to the $n$-th normalizer in $\Sym(2^n)$ of Eq.~\eqref{eq:chain}. 
The next contributions are rather technical and will really show the cost, in terms of combinatorial complexity, of trying to go beyond the `natural' limit of the \((n-1)\)-th idealizer/normalizer. 
 \medskip

Let 
\(x^\Lambda\partial_{k}\in                   \mathfrak{N}_{n}\setminus
\mathfrak{N}_{n-1}\) and let \(e_1  <  \dots <  e_s\)  be the excludants  of
\(\Lambda\).   We start by giving some necessary conditions that \(x^\Lambda\partial_{k}\) has to satisfy since it belongs to \(\mathfrak{N}_{n}\setminus\mathfrak{N}_{n-1}\).

 By Theorem~\ref{thm:homogeneous_normalizers}, we have   \(\wt(\Lambda)\ge k+1\).
Suppose        first       that        \(\wt(\Lambda)=k+1\).        By
Theorem~\ref{prop:unrefinable} either \(\Lambda\)
is refinable or \(\Lambda\) is unrefinable and \(k \le n-e_1 \).
If  \(\Lambda\)  is  refinable,  then there exists a  partition \(\Gamma\) with
\(h=\wt(\Gamma)<k\),                     such                     that
\(\mathfrak{N}_{n-1}\ni     x^\Theta\partial_{k}=[x^\Lambda\partial_k,
x^\Gamma  \partial_h]\ne   0\);  the  partition  \(\Theta\)   is  then
unrefinable and the minimal excludant \(e\) of \(\Theta\) is such that
\(k>n-e\).  Suppose that \(\Lambda\)  satisfies the \(j\)-th excludant
condition with \(j\ge 1\).
Since  there  exists  an unrefinable  \(2\)-refinement  \(\Theta\)  of
\(\Lambda\)  obtained   replacing  a   part  \(\lambda_u\)   with  two
excludants  \(e_s\)   and  \(e_t\),  then we have  \(j\le   3\).   Moreover,  if
\(j\ge       2\),       then       the        commutator       element
\[x^\Xi \partial_{e_1+k}=   [x^\Lambda\partial_{k}, \fil_{e_1,k}]  =
x_{e_1}x^\Lambda\partial_{e_1+k}   \in   \mathfrak{N}_{n-1},\]   therefore
 \(x^\Lambda\partial_k\) satisfies the weak second excludant
condition  and the  partition  \(\Xi\) obtained  by \(\Lambda\)  by
filling  its  minimum  excludant  is an  unrefinable  partition.  This
implies that any  refinement of  \(\Lambda\) has  \(1\) in the
\(e_1\)-th  component. In  the more specific case  \(j=3\), the  same argument  applies
replacing  \(e_1\)  with  \(e_2\).   Thus  every  refinement
\(\Theta\)  of  \(\Lambda\) has  each of  the  \(e_1\)-th and  \(e_2\)-th
component set to \(1\). From Proposition~\ref{prop:refinement_length}, we have 
that if  \(j=3\), then  \(\lambda_{e_1+e_2}=1\) and  \(\Theta\) is  obtained
from \(\Lambda\) inserting \(0\) in the \((e_1+e_2)\)-th component and
\(1\)  in  the \(e_1\)-th  and  \(e_2\)-th  component of  \(\Lambda\). 
 A  similar  argument   shows  that  if  \(\Lambda\)  is
unrefinable,  then  it  has  to   satisfy  the  second  weak  excludant
condition.
Let us summarize the previous conditions as follows:
\begin{definition}\label{def:one_step_mx}
	The element \(x^\Lambda\partial_k\) satisfies the \emph{\(1\)-step excludant condition} if \(wt(\Lambda)=k+1\) and \(\Lambda\) satisfies  one of the following:
	\begin{enumerate}[(a)]
	\item \(\Lambda\)  is \(1\)-step refinable and it satisfies the first excludant
	condition, \label{item:zero}
	\item \label{item:one}  \(\Lambda\)  is \(1\)-step refinable and it satisfies the second excludant
	condition and every refinement \(\Theta\) is such that \(\theta_{e_1}=1\), 
	\item  \label{item:two}  \(\Lambda\) is \(1\)-step refinable and satisfies both the third excludant
	condition and the second  weak excludant condition, \(\lambda_{e_1+e_2}=1\),
	and the only refinement \(\Theta\) of \(\Lambda\) is such that \(x^\Theta=x_{e_1}x_{e_2}x^\Lambda /x_{e_1+e_2} \),
	\item \label{item:three} \(\Lambda\)  is
	unrefinable  and  it  has  to   satisfy  the  second  weak  excludant
	condition.
\end{enumerate} 
\end{definition}

We are now  left with  the case   \(\wt(\Lambda)\ge  k+2\).  If
\(\lambda_1=1\),                                                  then
\(   x^\Theta\partial_{k}    =   [x^\Lambda\partial_{k},\partial_1]\in
\mathfrak{N}_{n-1}\),                      and so
\(k+1     \le     \sum_{i\ge    2}i\lambda_i\le     k+1\)     implies
\(\wt(\Lambda)=k+2\).  The minimal excludant of \(\Theta \) is
\(1\), which implies \(k>n-1\), i.e.\ \(k=n\). Moreover, \( \Theta \) has to
be   unrefinable  and so if   \(\lambda_i=\theta_i=0\)  for   some
\(i\ge  2\),  then  \(\lambda_{i+1}=\theta_{i+1}=0\)  as  well.   This
implies that there exists an index \(t\) such that \(\lambda_i=1\) for
\(1\leq i \leq t\) and \(\lambda_i=0\) for \(i>t\). Thus \(\Lambda\) is a
triangular partition.
Suppose now that \(\lambda_1=0\). Let \(h\)  be  an index such that
\(\lambda_{h-1}=0\) and \(\lambda_h=1\). We want to show that \(h=2\).
If       \(h>2\),      then       the        commutator       element
\(x^\Theta\partial_{k}=[x^\Lambda\partial_{k},        \fil_{1,h-1}]\in
\frak{N}_{n-1}\)  where \(\wt(\Theta)  =k+2\), a  contradiction.  This
implies that there exists an index \(t>2\) such that \(\lambda_i=0\) for
\(i>t\).   We will then say that  \(\Lambda\)  is   a  \emph{weak-triangular  partition}.   In
particular   \(\lambda_2=1\)    and   so   the    commutator   element
\([x^\Lambda\partial_{k},   x_1\partial_{2}]=  x^\Theta\partial_{k}\in
\frak{N}_{n-1}\), where the minimum  excludant of \(\Theta\) is \(2\).
Thus \(n-2< k  \le n\), i.e.\ \(k\) is either  \(n\) or \(n-1\).  Note
that     the case    \(k=n-1\)     cannot occur    since     then
\([x^\Lambda\partial_{k},  \fil_{1,k}]=  x^\Theta\partial_{n}\ne  0\in
\frak{N}_{n-1}\)       which       yields      the       contradiction
\(\wt(\Theta)=k+3=n+2>n+1\).\\

We conclude summarizing below what previously discussed and showing that the mentioned conditions
are also sufficient, with some sporadic exceptions in the case $n=8$. Due to the 
intricate combinatorial nature of the problem, the long proof of the result is rather tedious as it is articulated in several cases
and sub-cases.

\begin{theorem}\label{thm:nth}
  With the sole exclusion of the cases \(n=8\) and
  \begin{align}
  	x^\Lambda\partial_{k}  &= x_2x_7\partial_8, \label{excl1} \\
  	x^\Lambda\partial_{k}  &= x_4x_5\partial_8, \label{excl2}\\
  	x^\Lambda\partial_{k}  &= x_2x_4\partial_5, \label{excl3}
  \end{align}
the     element     \(x^\Lambda\partial_{k}    \)     belongs     to
  \(\mathfrak{N}_{n}\setminus \mathfrak{N}_{n-1}\) if  and only if 
  one
  of the following conditions is satisfied:
  \begin{enumerate}
  	\item \label{item1thm} \(\wt(\Lambda)=k+1\)  and \(x^\Lambda\partial_k\) satisfies the \emph{\(1\)-step excludant condition},
  \item \label{item2thm} \(\wt(\Lambda)=k+2\),  \(k=n\) and one of the following holds,  
  \begin{enumerate}[(i)]
  	\item \label{item2ithm} \(n+2\) is  the
  	\(t\)-th triangular number  and \(x^\Lambda=x_1\cdots x_t\), i.e.\
  	\(\Lambda\) is the \(t\)-th triangular partition,
  	\item \label{item2iithm} \(n+3\) is  the
  	\(t\)-th triangular number  and \(x^\Lambda=x_2\cdots x_t\), i.e.\
  	\(\Lambda\) is the \(t\)-th weak-triangular partition.
  \end{enumerate}
  \end{enumerate}
\end{theorem}
	
\begin{proof}
We            already            showed           that            if
  \(x^\Lambda\partial_{k}         \in        \mathfrak{N}_{n}\setminus
  \mathfrak{N}_{n-1}\)  then one  of  the conditions  \ref{item1thm}--\ref{item2thm} has to  be
  satisfied.  \smallskip
  
  We now  show that these  conditions are also sufficient.  The proof is divided in two steps. We prove, in \textsc{Step~1},                          
   that                           if
  \(x^\Gamma\partial_{u}\in                \mathfrak{N}_{n-1}\setminus
  \mathfrak{N}_{n-2}\),                                           then
  \( [x^\Lambda\partial_{k}, x^\Gamma\partial_{u}]=0\)  for every base
  element \(x^\Lambda\partial_{k}\) satisfying \ref{item1thm} or \ref{item2thm}. Later, by Remark~\ref{rem:N(n-2)}, it remains to prove that   if
\(\wt(x^\Gamma\partial_{u}) \le u\),                                            then
\( [x^\Lambda\partial_{k}, x^\Gamma\partial_{u}]  \in \mathfrak{N}_{n-1} \)  for every base
element \(x^\Lambda\partial_{k}\) satisfying \ref{item1thm} or \ref{item2thm}. This is accomplished in \textsc{Step~2}.

\medskip

\mybox{\textsc{Step~1.}}

  We  consider several  cases depending on  the form of the unrefinable
  partition \(\Gamma\). We note that, since \(\Gamma\) is not the zero
  partition, at  least  one  of \(\gamma_1\)  and
  \(\gamma_2\) has  to be not  zero or \(\Gamma\) would  be refinable.
  In      the      following       we      shall      assume      that
  \( [x^\Lambda\partial_{k}, x^\Gamma\partial_{u}]\ne 0\)  without further  mention, in  particular, for  all \(i\),  the
  components \(\lambda_i\) and \(\gamma_i\) will  be not both equal to
  \(1\).
  
  \textsc{Case 1.}   We start considering the  case \(\gamma_1=0\) and
  \(\gamma_2=1\).
  
 \noindent  If \(\gamma_{i-1}=0\)  and \(\gamma_i=1\)  for some \(i\ge  0\), then
  \(\Gamma\)   would  admit   a  refinement   \(\Theta\)  by   setting
  \(x^\Theta\partial_u=[x^\Gamma\partial_u,x_1x_{i-1}\partial_i]\).
  This is not possible as \(\Gamma\) is unrefinable and so \(\Gamma\)
  is    weak-triangular, i.e.\   \(x^\Gamma=    x_2\cdots   x_s\).     Since
  \(x^\Gamma\partial_{u}\in                \mathfrak{N}_{n-1}\setminus
  \mathfrak{N}_{n-2}\),  then  \(n-e<u\le  n\), where  \(e=1\)  is  the
  minimum  excludant of  \(\Gamma\). Thus  \(u=n\) and we  necessarily have \(k<n\).  Since the above commutator
  is not    trivial, then  \(2\le k \le s\).    As   a   consequence  \(\lambda_1=1\)   and
  \(\lambda_i=0\)         for        \(i\ge         2\),        giving
  \(x^\Lambda\partial_{k}=x_1\partial_{k}\in      \mathcal{U}\setminus
  \mathcal{T}\), a contradiction.
  
  \textsc{Case  2.}   We next  consider  the  case \(\gamma_1=1\)  and
  \(\gamma_2=\gamma_3=0\).
  
  \noindent Note  that  \(\gamma_4=1\)  otherwise the unrefinability  of  \(\Gamma\)
  would   give   \(\gamma_i=0\)   for    all   \(i\ge   2\).      If  \(\gamma_i=0\)   for  all  \(i  >  4\),  then
  \(u=\wt(\Gamma)   -1  =4\)   contradicting  \(\gamma_4=1\).    Hence
  \(\gamma_t\ne  0\)  for  some  \(t >  4\).   Note  that  necessarily
  \(\gamma_i=0\) if \(i>3\) and \(i\equiv  3 \bmod 2\), i.e.\ if \(i\)
  is  odd.   Thus \(t=2h\)  for  some  \(h\ge  3\).  If  \(h>3\),  then
  \(2h=(2h-3)+3\) and the fact that \(\gamma_3=\gamma_{2h-3}=0\) shows
  that                          the                         commutator
  \(x^\Theta\partial_{u}= [x^\Gamma\partial_u,x_3x_{t-3}\partial_t]\)  provides  a
  proper refinement \(\Theta\) of  \(\Gamma\), a contradiction.  Hence
  \(h=3\)      is      the      only     possible      choice     and so
  \(x^\Gamma\partial_{u}=x_1x_4x_6\partial_{10}\).                From
  \(n-e<u\le  n\),   where  \(e=2\)   is  the  minimum   excludant  of
  \(\Gamma\),   we   have   \(n=10\)  or   \(n=11\).    Suppose  first that
  \(k<u=10\le  n\). Then the  possible values  of \(k\)  are \(4\) or \(6\),  
  otherwise   \(  [x^\Lambda\partial_{k},   x^\Gamma\partial_{u}]=0\).
  Since either  \(\wt(\Lambda)=k+1\) or \(\wt(\Lambda)=k+2\),  the only
  possibilities         for        \(x^\Lambda\partial_k\)         are
  \(x^\Lambda\partial_k=x_2x_5\partial_6\),
  \(x^\Lambda\partial_k=x_3x_5\partial_6\)                         or
  \(x^\Lambda\partial_k=x_2x_3\partial_4\).   None of  these satisfies
  \ref{item1thm} or \ref{item2thm}, since we are assuming \(n\ge 10\).
  The other possibility  is that \(k>u=10\). As  a consequence we have \(k=n=11\)
  and \(\lambda_{10}=1\).  Since \(\wt(\Lambda)\le k+2 \le 13\), then \(\Lambda\) cannot be weak-triangular.
  Hence     \(\wt(\Lambda)=k+1\)    and     we    consequently have     
  \(x^\Lambda\partial_k=x_2x_{10}\partial_{11}\),   which does   not
  satisfy any of the conditions of Definition~\ref{def:one_step_mx}, as \(\Lambda\) is \(2\)-step refinable.
  
  \textsc{Case 3.}   We consider the case  \(\gamma_1=\gamma_3=1\) and
  \(\gamma_2=0\).
  
 \noindent The unrefinable partition \(\Gamma\) has then the form:
  \begin{equation}\label{eq:gamma}
    \Gamma = (1,0,\underbrace{1,\dots,1}_{l}, \underbrace{0,1,\dots,0,1}_{m},0,\dots). 
  \end{equation}
  In other words, the sequence  starts with  \(1,0\) followed  by \(l\ge 1\)  repetitions of
  \(1\) and by  \(m\ge 0\) repetitions of the block \(0,1\),   and then it is
  definitely \(0\). Correspondingly, \(\Lambda\) has the form
  \[
    \Lambda             =             (0,*,\underbrace{0,\dots,0}_{l},
    \underbrace{*,0,\dots,*,0}_{m},*,\dots),
  \]
  where  the asterisks  are  unspecified  digits  in  \(\Set{0,1}\).   Since
  \(\Lambda\notin \mathcal{U}\)  then it  has at least  two components
  equal to \(1\).  In particular \(\lambda_i=1\) for some  minimal \(i> l+2\). 
  We proceed by considering the possible values of \(i\).  
  
  Suppose first that \(i\ge 8\), which gives \(\lambda_4=0\).  In  this case 
  \(\lambda_{i-3}=0\) and          the          commutator
  \(x^\Theta\partial_k=[x^\Lambda\partial_k,x_3x_{i-3}\partial_i]   \)
  provides a refinement \(\Theta\) of \(\Lambda\) which is unrefinable
  since            the           commutator            is           in
  \(\mathfrak{N}_{n-1}\setminus       \mathfrak{N}_{n-2}\).       Moreover
  \(\theta_1=0\), and so  \(\Theta\)  has to be  weak-triangular. This
  implies  that   \(\lambda_4=\theta_4=1\), a contradiction. Therefore   \(i<8\).
 
   If \(i=7\), then  \(\lambda_4\ne 0\) otherwise
  the                                                       commutator
  \(x^\Theta\partial_k=[x^\Lambda\partial_k,x_3x_4\partial_7]      \in
  \mathfrak{N}_{n-1}\setminus\mathfrak{N}_{n-2}\)   would  provide   a
  weakly     triangular      partition     \(\Theta\). The fact \(i=7\) also implies     \(\lambda_5=0\).
  Hence
  \(x^\Theta\partial_k=x_2x_3x_4\partial_{8}\),
  \(n=k=8\)                                             and
  \(x^\Lambda\partial_k=x_2x_7\partial_{8}\), which is the sporadic exception of Eq.~\eqref{excl1}.  
   
  Suppose that \(i=6\), then \(\gamma_6=0\).
  The unrefinability of \(\Gamma\) implies that  \(\gamma_{2h}=0\)  for  \(h\ge  0\)  with  the  only
  possible exception of  \(h=2\).  Let \(s\) be  largest possible such
  that     \[\gamma_1=\gamma_3=\dots=\gamma_{2s-1}=     1.\]     Then
  \(\gamma_j=0\)      for      all       \(j>2s-1\).      We      have
  \[\wt(\Gamma)=1+3+\dots   +  2s-1+4\gamma_4=s^2+4\gamma_4,\]  hence
  \(n-2<u=s^2-1+4\gamma_4 \le    n\).    Thus
  \(s^2-1+4\gamma_4\le    n    \le   s^2+4\gamma_4\).     Note    that
  \(\lambda_j=0\) for all odd  \(j\le 2s-1\).  Suppose \(\lambda_t=0\)
  for all \(t\ge 2s\). Since \(i=6\), then \(\lambda_4=0\) and so \(k+1=\wt(\Lambda)  \le 2+6+\dots+  2s-2=s^2-s-4\).
  From this we obtain \(k\le s^2-s-5\). By hypothesis \(x^\Lambda\partial_{k}\) satisfies Definition~\ref{def:one_step_mx}, and so \(k > n-e_1-e_2=n-4\ge s^2 -5\). This gives \(s^2-5 < s^2-s-5 \), a contradiction.
  Hence \(\lambda_t=1\) for
  some minimum  \(t\ge 2s\).   
  We consider first the case \(s>3\) and so \(t\ge 2s\ge 8\).
  Suppose that \(t-3\ge 2s\), then
  \(\lambda_3=\lambda_{t-3}=0\).  If \(t-3 <  2s\) and \(t\)
  is even,  then \(t-3>3\)  is odd  and less than  \(2s\), hence  we find
  again \(\lambda_3=\lambda_{t-3}=0\). We are left with the case \(t\)
  odd and \(3 < t-3  < 2s \), i.e.\  \(t=2s+1\). If \(\lambda_{2(s-1)}=0\),
  then again \(\lambda_3=\lambda_{t-3}=0\).  Otherwise 
  we can  refine \(\Gamma\)  two times by  replacing \(2s-2=(2s-3)+1\)
  and then  \(t=(2s-2)+3\) which is  impossible.  Summarizing, if  \(s\ge 4\)
  then \(t-3\ne 3\) and   \(\lambda_3=\lambda_{t-3}=0\). If  follows that
  \(\Lambda\)  can be  refined once  by replacing  \(t\) by  \(3\) and
  \(t-3\) obtaining a weak triangular  partition. This can happen only
  if \(t=8\) and
  \(x^\Lambda\partial_k=x_2x_4x_6x_8\partial_{19}\).
  If \(k<u\) then \(\gamma_{19}\ne 0\) so that  \(u>2s-1\ge 19\). In particular \(8 = t\ge 2s \ge 20\) a contradiction. Hence \(k>u\). In this case, since \(s\ge 4\), we have \(u> 2s-1\ge 7\) and in order to have \([x^\Lambda\partial_k,x^\Gamma\partial_u]\ne 0\) we have \(u=8\), \(\gamma_4=0\) and so \(x^\Gamma=x_1x_3x_5x_7\). We then have \([x^\Lambda\partial_k,x^\Gamma\partial_u]= x_1x_2x_3x_4x_5x_6x_7x_8\partial_{19}\) and \(20=k+1< 1+2+\dots+8=28\), again a contradiction.
  Thus we may  assume \(s\le 3\)
  which                                                        implies
  \(x^\Gamma\partial_u=x_1x_3x_4^{\gamma_4}x_5\partial_{8+4\lambda_4}\)
  and      \(n=8+4\gamma_4\)      or     \(n=9+4\gamma_4\).       Correspondingly,
  \(x^\Lambda\partial_k=x_2^{\lambda_2}x_4^{\lambda_4}x_6x_8^{\lambda_8}\partial_k\). Note that if \(\lambda_{8}=1\), then \(\lambda_2=\lambda_4=0\) otherwise \(k+1=\wt(\Lambda)> 14 > n\) which gives the contradiction \(k\ge n+1\), hence \(x^\Lambda=x_6x_8\) and \(\Lambda\) is \(2\)-step refinable. Thus \(\lambda_8=0\).  
  Since  \(k>6\),   we  have   that  \(x^\Gamma\partial_u\)   and
  \(x^\Lambda\partial_k\)  commute, a contradiction.
   
  We        now        suppose        that        \(i=5\). In this case \(\lambda_4=0\). Assuming  \(\lambda_2=0\) would give that \(\Lambda \) is \(2\)-step refinable.         Then
  \(x^\Lambda=x_2x_5\prod_{i\ge 6}x_i^{\lambda_{i}}\) and \(k=6+\sum_{i\ge 6}i\lambda_i\). We first note that \(l=2\) and \(m\le 3\) since \(\Gamma\) is unrefinable and \(5\) and \(7\) are excludants. Also \(n-1\le \wt(\Gamma)-1 = m^2+5m+7 \le n\) which in turn gives \[
  n=\begin{cases}
  	m^2+5m+7 \\
  	m^2+5m+8
  \end{cases}.
  \]
  Note that \(x^\Lambda\partial_k\) satisfies the second weak excludant condition, hence 
  \[
  m^2+5m+6 \le n-1 \le \wt(\Lambda) \le n+1 \le  m^2+5m+9.
  \]
  We proceed analyzing the possible values of \(m\). If \(m=0\) then \(6 \le \wt(\Lambda) \le 9\). The only possibility is that \(x^\Lambda\partial_{k}=x_2x_5\partial_{6}\). Then  \(\Lambda\) is refinable since \(5=1+4\) and its only refinement has to satisfy the first  excludant condition i.e.\   \(n-3 < 6\) and hence  \(n\le 8\). Noting that 
  \(x^\Gamma\partial_u=x_1x_3x_{6}\partial_{9}\) we also have  \(n\ge 9\), which
  is    a    contradiction. Let us now assume \(m=1\). In this case \(12 \le \wt(\Lambda) \le 15\). The only possibilities for \(x^\Lambda\partial_{k}\) are \(x_2x_5x_7\partial_{13}\) and \(x_2x_5x_8\partial_{14}\). The second one is \(2\)-step refinable and so it is not eligible. If \(x^\Lambda\partial_{k}=x_2x_5x_7\partial_{13}\), then \(x^\Gamma\partial_{u}=x_1x_3x_4x_6\partial_{13}\) and these two elements commute. Let us now consider the case \(m=2\).  We have \(20 \le \wt(\Lambda) \le 23\) giving the only possibility \(x^\Lambda\partial_{k}=x_2x_5x_7x_9\partial_{22}\) which is \(2\)-step refinable. Similarly when \(m=3\) we find  \(30 \le \wt(\Lambda) \le 33\) which in turn yields that the possible cases for \(x^\Lambda\partial_{k}\) are represented by \(2\)-step refinable partitions.
  
  Suppose that \(i=4\), then \(\gamma_4=0\).
  The unrefinability of \(\Gamma\) implies that  \(\gamma_{2h}=0\)  for  \(h\ge  0\).  Let \(s\ge 2\) be  largest possible such
  that     \(\gamma_1=\gamma_3=\dots=\gamma_{2s-1}=     1\).      Then
  \(\gamma_j=0\)      for      all       \(j>2s-1\).      We      have
  \(\wt(\Gamma)=1+3+\dots   +  2s-1=s^2\),   hence
  \(n-2<u=s^2-1\le    n\).    
  Thus
  \(s^2-1\le    n    \le   s^2\).     Note    that
  \(\lambda_j=0\) for all odd  \(j\le 2s-1\).  Suppose \(\lambda_t=0\)
  for all \(t\ge 2s\). Since \(i=4\) then   \(k+1=\wt(\Lambda)  \le 2+4+\dots+  2s-2=s^2-s\).
  From this we obtain  \(k\le s^2-s-1\). By hypothesis \(x^\Lambda\partial_{k}\) satisfies Definition~\ref{def:one_step_mx}, so that \(k > n-e_1-e_2=n-4\ge s^2 -5\). This gives \(s^2-5 < s^2-s-1 \) and so \(2\le s\le 3\).
  If \(s=2\), then \(k=3\) and so  \(\gamma_3\) has to be \(0\), contrary to the hypothesis. If \(s=3\), then \(x^\Gamma\partial_u=x_1x_3x_5\partial_8\) and therefore \(n=8\) or \(n=9\). Since \(x^\Lambda\partial_k\) satisfies the second weak excludant condition, then \(5\le n-3\le k\le n\le 9\), so we have    \(x^\Lambda=x_2^{\lambda_2}x_4x_6^{\lambda_6}\). Since \([x^\Lambda\partial_k,x^\Gamma\partial_u]\ne 0\), then \(k=5\) and so \(\lambda_2=1\), \(\lambda_6=0\) and necessarily \(n=8\). In this way we obtain the exceptional element of Eq.~\eqref{excl3}.
 We are left with the case \(\lambda_t=1\) for
  some minimum  \(t\ge 2s\).   
  We consider first the case \(s>3\) and so \(t\ge 2s\ge 8\).
  Suppose that \(t-3\ge 2s\), then
  \(\lambda_3=\lambda_{t-3}=0\).  If \(t-3 <  2s\) and \(t\)
  is even, then \(t-3>3\)  is odd  and less than  \(2s\), hence  we find
  again \(\lambda_3=\lambda_{t-3}=0\). We are left with the case \(t\)
  odd and \(3 < t-3  < 2s \), i.e.\  \(t=2s+1\). If \(\lambda_{2(s-1)}=0\),
  then again \(\lambda_3=\lambda_{t-3}=0\).  Otherwise 
  we can  refine \(\Gamma\)  two times by  replacing \(2s-2=(2s-3)+1\)
  and then  \(t=(2s-2)+3\) which is  impossible.  Summarizing, if  \(s\ge 4\),
  then \(t-3\ne 3\) and   \(\lambda_3=\lambda_{t-3}=0\). If  follows that
  \(\Lambda\)  can be  refined once  by replacing  \(t\) by  \(3\) and
  \(t-3\) obtaining a weak triangular  partition. This can happen only
  if \(t=8\) and
  \(x^\Lambda\partial_k=x_2x_4x_6x_8\partial_{19}\).
  If \(k<u\), then \(\gamma_{19}\ne 0\) so that  \(u>2s-1\ge 19\). In particular \(8 = t\ge 2s \ge 20\), a contradiction. Hence \(k>u\). In this case, since \(s\ge 4\), we have \(u> 2s-1\ge 7\) and in order to have \([x^\Lambda\partial_k,x^\Gamma\partial_u]\ne 0\) we have \(u=8\), \(\gamma_4=0\) and so \(x^\Gamma=x_1x_3x_5x_7\). Then we have \([x^\Lambda\partial_k,x^\Gamma\partial_u]= x_1x_2x_3x_4x_5x_6x_7x_8\partial_{19}\) and \(20=k+1< 1+2+\dots+8=28\), again a contradiction.
  Thus we may  assume \(s\le 3\)
  which                                                        implies
  \(x^\Gamma\partial_u=x_1x_3x_5\partial_{8}\)
  and      \(n=8\)      or     \(n=9\).       Also
  \(x^\Lambda=x_2^{\lambda_2}x_4x_6^{\lambda_6}x_8^{\lambda_8}\). Note that if \(\lambda_{8}=1\) then \(k+1=\wt(\Lambda)\ge 12\) which contradicts \(n\le 9\).
Thus \(\lambda_8=0\). We also have  \(\lambda_6=0\), otherwise \([x^\Lambda\partial_k,x^\Gamma\partial_u]=0\). We have again obtained the sporadic exception 
\(x^\Lambda\partial_k=x_2x_4\partial_5\) and \(n=8\).

  \textsc{Case   4.}     We   are    finally   left   with    the   case
  \(\gamma_1=\gamma_2=1\).    
  
  \noindent We have \(\lambda_1=\lambda_2=0\). 
 Suppose
  that  \(\lambda_i=0\)  and   \(\lambda_{i+2}=1\)  for  some  minimal
  \(i\ge                           3\),                           then
  \(x^\Theta\partial_k=[x^\Lambda\partial_k,x_2x_{i}\partial_{i+2}]\in
  \mathfrak{N}_{n-1}\setminus       \mathfrak{N}_{n-2}\).        Since
  \(\theta_1=0\), then \(\Theta\) has to be weak triangular. Hence
  \[ \Lambda=(0,0,\underbrace{1,\dots,1}_{i-3},0,*,1,0,\dots )
  \]
and correspondingly
  \[
    \Gamma=(1,1,\underbrace{0,\dots,0}_{i-3},*,*,0,*,\dots ),
  \]
    where  the asterisks  are  unspecified  digits  in  \(\Set{0,1}\). 
  We have either  \(x^\Gamma=x_1x_2\), which would imply \(n=2\), or \(i\le 6\); in the latter case we have that \(\Gamma\) is refinable.
  
  Suppose first  that \(i=6\). Again, the  unrefinability of \(\Gamma\)
  gives  \(\gamma_6=1\)  and \(\gamma_7=\gamma_8=\gamma_9=0\). Since \(\gamma_3=0\), unrefinability of \(\Gamma\) implies that \(\gamma_j=0\) for \(j\ge 10\).   Hence
  \(x^\Gamma\partial_u=   x_1x_2x_6\partial_8\)    and   \(8\le   n\le
  10\). Also \(n\ge k=\wt(\Lambda)-1\ge 3+4+5+8>10\), a contradiction.
  
  Suppose now that \(i=5\),  hence \(\lambda_5=0\) and \(\lambda_7=1\)
  and
  \(\gamma_3=\gamma_4=\gamma_7=\gamma_{10}=\gamma_{11}=\gamma_{13}=\gamma_{14}=\gamma_{15}=0\). We
  have
  \(x^\Lambda\partial_k=
  x_3x_4x_6^{\lambda_6}x_7\partial_{13+6\lambda_6}\),   in  particular
  \(n\ge 13\).   The unrefinability of \(\Gamma\) gives  \(\gamma_{j}= 0\)
  for \(j\ge 13\).  Also
  \[x^\Gamma\partial_u=x_1x_2x_5^{\gamma_5}x_6^{\gamma_6}x_8^{\gamma_8}x_9^{\gamma_9}x_{12}^{\gamma_{12}}\partial_{2+5\gamma_5+6\gamma_6+8\gamma_8+9\gamma_9+12\gamma_{12}}.\]
  If   \(x^\Gamma\partial_u\)  and   \(x^\Lambda\partial_k\)  do   not
  commute, then \(u\) has to be  \(3\), \(4\), \(6\) or \(7\). This is
  possible only if  \(x^\Gamma\partial_u=x_1x_2x_5\partial_7\) and
  \(7\le n  \le 9\), which is  incompatible with the  previously computed
  bound for \(n\).
  
  Assume now  that \(i=4\),  so that  \(\lambda_4=0\), \(\lambda_6=1\)
  and              \(\gamma_6=0\).                We              have
  \(x^\Lambda\partial_k=x_3x_5^{\lambda_5}x_6\partial_{9+5\lambda_5}\),
  yielding     \(n=9\)     or      \(n=14\).      In     this     case
  \[x^\Gamma\partial_u=
  x_1x_2x_4^{\gamma_4}x_5^{\gamma_5}x_7^{\gamma_7}x_8^{\gamma_8}x_{10}^{\gamma_{10}}x_{11}^{\gamma_{11}}
  x_{13}^{\gamma_{13}}x_{14}^{\gamma_{14}}\cdots \partial_u. \] These
  two elements   commute unless
  \(k=14\)  and \(\gamma_{14}=1\), in  which case \(n> u>  16\), a contradiction, or \(u=6\) and so \(n\le 8\), again a contradiction.

  We   are  then   left  to   consider  \(i=3\).    In  this   case
  \(x^\Lambda\partial_k=x_4x_5\partial_{8}\)  and  therefore   \(n=8\).   Correspondingly
  \(x^\Gamma\partial_u=x_1x_2x_3^{\gamma_3}x_6^{\gamma_6}\cdots
  \partial_u\).  These two elements do not commute  if  \(x^\Gamma\partial_u=x_1x_2x_3\partial_5\)         and
    \(x^\Lambda\partial_k=x_4x_5\partial_{8}\), which is the exceptional element of Eq.~\eqref{excl2}, or \(\gamma_8=1\),    which   implies    \(\gamma_3=1\)    and
    \(x^\Gamma\partial_u=x_1x_2x_3x_6^{\gamma_6}x_7^{\gamma_7}x_8\cdots
    \partial_u\)   and so   \(n\ge  u=\wt(\Gamma)-1>   13>n=8\),   a
    contradiction.
    
\medskip

\mybox{\textsc{Step~2.}}

We now proceed by showing that   if
\(\wt(\Gamma)\le u\),                                        then
\( [x^\Lambda\partial_{k}, x^\Gamma\partial_{u}]  \in \mathfrak{N}_{n-1} \)  for every base
element \(x^\Lambda\partial_{k}\) satisfying \ref{item1thm} or \ref{item2thm}. 
Let us denote \(x^\Theta\partial_v= [x^\Lambda\partial_{k}, x^\Gamma\partial_{u}]  \) and let us assume $x^\Theta\partial_v\ne 0$.

We treat first the case \(k<u\) which means \(v=u\). This implies \(\gamma_k=1\) and \(x^\Gamma= x_k\prod x_{e_i}^{\gamma_{e_i}}\), where the \(e_i\) are the excludants of \(\Lambda\) and \(u\le k+\sum{\gamma_{e_i}}\le n\).  If \(x^\Gamma=x_k\), then  \(x^\Theta\partial_v=x^\Lambda\partial_u\) and so \(\wt(\Lambda)=u+1=k+1\) or \(\wt(\Lambda)=u+1=k+2\). Since \(k<u\), we have \(u=k+1\) and \(x^\Lambda\partial_k\) satisfies  \ref{item2thm}\ref{item2ithm} or \ref{item2thm}\ref{item2iithm}; in particular \(k=n\) and \(u=n+1\), a contradiction. We may then assume that \(\gamma_{e_i}=1\) for some \(i\).
Let us suppose that \(x^\Lambda\partial_k\) satisfies Definition~\ref{def:one_step_mx}\ref{item:zero}.  Then 
\( \wt(\Gamma)=k+\sum\gamma_{e_i} \ge k+e_1 > n\), a contradiction.
If \(x^\Lambda\partial_k\) satisfies Definition~\ref{def:one_step_mx}\ref{item:one}, then \(x^\Gamma=x_{e_1}x_k\), otherwise \(u\ge \wt(\Gamma) =  k+\sum{\gamma_{e_i}} \ge k+e_2 >n  \), and in this case \(x^\Theta\partial_u\in \mathfrak{N}_{n-1}\) as \(\Theta\) is unrefinable and its minimal excludant is \(e_2\).
Let us now suppose that \(x^\Lambda\partial_k\) satisfies Definition~\ref{def:one_step_mx}\ref{item:two} or Definition~\ref{def:one_step_mx}\ref{item:three}. Since we are assuming \(0\ne  [x^\Lambda\partial_{k}, x^\Gamma\partial_{u}]  \), we have  \(x^\Gamma=x_{e_1}x_k\) or \(x^\Gamma=x_{e_2}x_k\). Otherwise, as above, we have \(u>n\) since we are assuming the second weak excludant condition for \(x^\Lambda\partial_{k}\). Also in this case a direct check of \(\wt(\Theta)\) shows that \(x^\Theta\partial_u\in \mathfrak{N}_{n-1}\). 
Finally, if \(x^\Lambda\partial_k\) satisfies \ref{item2thm}\ref{item2ithm} or \ref{item2thm}\ref{item2iithm}, then \(u\ge k+1 >n\) which is not possible.

To conclude, let us consider the case \(k>u\). In this case \(v=k\) and either  \(\Theta\) is a refinement of \(\Lambda\) or \(\wt(\Theta) < \wt(\Lambda)\). This implies that  \(  [x^\Lambda\partial_{k}, x^\Gamma\partial_{u}]  \in \mathfrak{N}_{n-1} \) if \(x^\Lambda\partial_k\) satisfies Definition~\ref{def:one_step_mx}\ref{item:three} or \ref{item2thm}\ref{item2ithm} or \ref{item2thm}\ref{item2iithm}. If otherwise \(x^\Lambda\partial_k\) satisfies Definition~\ref{def:one_step_mx}\ref{item:zero}\ref{item:one}\ref{item:two},
then either \(\Theta\) is unrefinable of weight \(k+1\) and satisfy the minimal excludant condition, or \(\wt(\Theta)\le v=k \). In both cases \(x^\Theta\partial_v\in \mathfrak{N}_{n-1}\). This concludes the proof.
\end{proof}

\section{Conclusions and open problems}\label{sec:conclusion}
Computing the chain of normalizers of Eq.~\eqref{eq:chain} is a computationally challenging task which soon clashes with the exponential growth of the order of the considered groups. In fact, as already pointed out in~\cite{Aragona2020}, computing the chain of normalizers up to the $(n-2)$-th term and more would not have been possible without introducing rigid commutators~\cite{Aragona2021}. Unfortunately, it appears that there is no \emph{natural} way to generalize the notion of rigid commutators when $p$ is odd in such a way these turn out  to be closed under commutation. An odd version of the rigid commutator machinery, as described in the cited paper for \(p=2\), would be indeed the key ingredient that could prove helpful in computing the chain of normalizers in $\Sym(p^n)$. This task is otherwise computationally unfeasible when $p\geq 3$, even when minimal values of $n$ are considered.

With this goal in mind, in this work we have introduced  a new framework which moves the setting from the symmetric group to a Lie ring with a basis of elements represented by partitions of integers which parts can be repeated no more than $m-1$ times. In this framework, the construction of the Lie ring reflects the construction of the Sylow \(p\)-subgroup of \(\Sym(p^n)\) when $m=p$ is prime, and still provides meaningful results when $m$ is composite. We defined the corresponding idealizer chain in the Lie ring and proved, as expected, that the growth of the idealizer chain goes as in the case of $\Sym(2^n)$ when $m=2$, and proceeds according to its natural generalization when $m > 2$ (cf.\ Theorem~\ref{cor:distinct_parts} and Theorem~\ref{cor:main}). In particular, when $m=2$ an explicit bijection between generators which preserves commutators is provided (cf.\ Definition~\ref{def:bije} and Theorem~\ref{thm:liechan}).

The possible obvious extensions of the notion of rigid commutators in the case \(p\) odd, to which will correspond a bijection similar to that given in Definition~\ref{def:bije}, do not produce a set of commutators that turns out to be closed under commutation, a property that is crucial in the proof of Theorem~\ref{thm:liechan}. 
If a commutation-closed extension were found, it would not be hard to believe that a natural correspondence preserving commutators between the new rigid commutators and  the basis elements of the Lie ring, as the one of Definition~\ref{def:bije}, may exist.
This would imply that Theorem~\ref{cor:main} is the $p$-analog of the chain of normalizers in $\Sym(p^n)$, where $m=p$ is odd, which at the time of writing remains a very plausible conjecture for which this paper, in the absence of any computational evidence, represents a source of support.
\bibliographystyle{amsalpha}
\bibliography{sym2n_ref}

\end{document}